\documentclass[oneside,english]{amsart}
\usepackage[T1]{fontenc}
\usepackage[latin9]{inputenc}
\usepackage{amsthm}
\usepackage{amsmath}
\usepackage{amssymb}
\usepackage{setspace}
\usepackage{xargs}[2008/03/08]
\usepackage[all]{xy}
\makeatletter
\numberwithin{equation}{section}
\numberwithin{figure}{section}
\usepackage{enumitem}		
 \@addtoreset{subsubsection}{section}
\theoremstyle{plain}
\newtheorem{thm}[subsubsection]{Theorem}
  \theoremstyle{remark}
  \newtheorem{notation}[subsubsection]{Notation}
  \theoremstyle{definition}
  \newtheorem{defn}[subsubsection]{Definition}
 \theoremstyle{definition}
  \newtheorem{example}[subsubsection]{Example}
  \theoremstyle{remark}
  \newtheorem{rem}[subsubsection]{Remark}
  \theoremstyle{plain}
  \newtheorem{lem}[subsubsection]{Lemma}
  \theoremstyle{plain}
  \newtheorem{prop}[subsubsection]{Proposition}
  \theoremstyle{plain}
  \newtheorem{cor}[subsubsection]{Corollary}

\usepackage{eucal}
\usepackage{url}
\usepackage{mathdots}
\usepackage[colorlinks=true, citecolor=blue]{hyperref}
\usepackage[all]{xy}

\newcommand{\xyR}[1]{
  \xydef@\xymatrixrowsep@{#1}}
\newcommand{\xyC}[1]{
  \xydef@\xymatrixcolsep@{#1}}

\makeatletter
\newtheorem*{rep@theorem}{\rep@title}
\newcommand{\newreptheorem}[2]{%
\newenvironment{rep#1}[1]{%
 \def\rep@title{#2 \ref{##1}}%
 \begin{rep@theorem}}%
 {\end{rep@theorem}}}
\makeatother

\newreptheorem{thm}{Theorem}

\makeatother

\usepackage{babel}
\begin{document}
\global\long\def\t#1{\mathrm{#1}}

\global\long\def\cart#1#2#3#4{\xymatrix{#1\ar[r]\ar[d]\ar@{}[dr]|(0.2)\ulcorner &  #3\ar[d]\\
#2\ar[r]  &  #4 
}
 }

\global\long\def\ran{\t{\mathcal{R}an}_{X}}
\global\long\def\dom{\t{Dom}_{X}}
\global\long\def\Ran{\t{Ran}_{X}}

\global\long\def\colim{\t{colim\,}}

\global\long\def\qco{\mathfrak{Qco}}
\global\long\def\indco{\mathfrak{Ico}}
\global\long\def\dmod{\mathfrak{Dmod}}

\global\long\def\spec{\t{spec}}
\global\long\def\mod{\t{mod}}

\global\long\def\fin{\mathfrak{Fin}}
\global\long\def\finop{\mathfrak{Fin}^{\text{op}}}
\global\long\def\finsurop{\mathfrak{Fin}_{\t{sur}}^{\t{op}}}
\global\long\def\finsur{\mathfrak{Fin}_{\t{sur}}}

\global\long\def\H{\mathcal{H}}

\global\long\def\S{\mathfrak{S}}
\global\long\def\Sop{\S^{\t{op}}}
\global\long\def\aff{\mathfrak{Aff}}
\global\long\def\affop{\mathfrak{Aff}^{\t{op}}}

\global\long\def\prl{\mathcal{P}r^{\t{\mbox{L}}}}
\global\long\def\prr{\mathcal{P}r^{\t R}}
\global\long\def\cathat{\hat{\mathbf{Cat}}_{\infty}}
\global\long\def\catex{\hat{\mathbf{Cat}}_{\infty}^{\t{Ex}}}

\global\long\def\catexl{\hat{\mathbf{Cat}}_{\infty}^{\t{Ex,L}}}
\global\long\def\catexr{\hat{\mathbf{Cat}}_{\infty}^{\t{Ex,R}}}
\global\long\def\catexlr{\hat{\mathbf{Cat}}_{\infty}^{\t{Ex,LR}}}

\global\long\def\set{\mathbf{Set}}
\global\long\def\cat{\mathbf{Cat}}
\global\long\def\gpd{\mathbf{Gpd}}
\global\long\def\catinf{\mathbf{Cat}_{\infty}}
\global\long\def\gpdinf{\mathbf{Gpd}_{\infty}}
\global\long\def\vect{\mathbf{Vect}}

\global\long\def\shv#1#2{\mathcal{S}hv_{#2}\left(#1\right)}
\global\long\def\pshv#1{\mathcal{P}shv\left(#1\right)}

\newcommandx\qmap[2][usedefault, addprefix=\global, 1=]{\t{QMap}^{#1}\left(X,#2\right)}
\newcommandx\gmap[2][usedefault, addprefix=\global, 1=]{\t{GMap}^{#1}\left(X,#2\right)}
\newcommandx\qsect[4][usedefault, addprefix=\global, 1=, 2=]{\t{QSect}_{#2}^{#1}\left(#3,#4\right)}
\newcommandx\gsect[4][usedefault, addprefix=\global, 1=, 2=]{\t{GSect}_{#2}^{#1}\left(#3,#4\right)}

\newcommandx\bun[1][usedefault, addprefix=\global, 1=]{\t{Bun}_{G}^{#1}}
\global\long\def\bunbbar#1{\overline{\t{Bun}_{G}^{B#1}}}
\global\long\def\bungbar#1{\overline{\t{Bun}_{G}^{#1}}}

\global\long\def\lke#1{\t{LKE}_{#1}}
\global\long\def\rke#1{\t{RKE}_{#1}}

\title{D-modules on Spaces of Rational Maps and on Other Generic Data}

\author{Jonathan Barlev}

\date{\today}
\begin{abstract}
Let $X$ be an algebraic curve. We study the problem of parametrizing
geometric data over $X$, which is only generically defined. E.g.,
parametrizing generically defined (aka rational) maps from $X$ to
a fixed target scheme $Y$. 

There are three methods for constructing functors of points for such
moduli problems (all originally due to Drinfeld), and we show that
the resulting functors are equivalent in the fppf Grothendieck topology. 

As an application, we obtain three presentations for the category
of D-modules ``on'' $B\left(K\right)\backslash G\left(\mathbb{A}\right)/G\left(\mathbb{O}\right)$,
and we combine results about this category coming from the different
presentations.
\end{abstract}
\maketitle
\tableofcontents

\section{Introduction}

Let $k$ be an algebraically closed field of characteristic 0, and
let $X$ be a smooth, projective and connected algebraic curve over
$k$. Denote by $\mathbb{A}$, $\mathbb{O}$ and $K$ the algebra
of adeles, algebra of integer adeles, and the field of rational functions
over $X$, respectively. 

In this paper we study the problem of parametrizing geometric data
over $X$ which is only generically defined. The basic example of
such a moduli problem is that of generically defined maps (rational
maps) from $X$ to a fixed target scheme $Y$. I.e., the starting
point is the given set of $k$-points (in this case it is the set
$\t{Hom}\left(\spec\left(K\right),Y\right)$) and the task at hand
is that of constructing a functor of points $\mathbf{Scheme}^{\t{op}}\rightarrow\mathbf{Set}$
which describes what an $S$-family of such generic maps is, for an
arbitrary scheme $S$. 

The main example we are interested in is motivated by the Langlands
program. In the classical setting, one encounters sets such as $B\left(K\right)\backslash G\left(\mathbb{A}\right)/G\left(\mathbb{O}\right)$,
$N\left(K\right)\backslash G\left(\mathbb{A}\right)/G\left(\mathbb{O}\right)$
and their relatives. The premise of the geometric program is that
these sets are the $k$-points of some space (``space'' interpreted
very loosely). The story goes that each point in the set $B\left(K\right)\backslash G\left(\mathbb{A}\right)/G\left(\mathbb{O}\right)$
is to be interpreted as representing a $G$-bundle on $X$, together
with the data of a reduction to $B$ at the generic point of $X$%
\footnote{This is admittedly equivalent to the set of global reductions to $B$,
but this interpretation leads to a different space! See also remark
\ref{rem:families of generic maps}.%
}. We wish to describe a space parametrizing such data via a functor
of points, and as above our starting point is the given set of $k$-points,
and our task is to define what an $S$-family of such generic reductions
is. 

The literature (and mathematical folklore) contains three, a-priori
different, construction schemas for such moduli problems (all originally
due to Drinfeld). The main result of this paper is that all three
constructions give rise to functors of points which are equivalent
in the fppf Grothendieck topology. Consequently, various invariants
such as the categories of quasi-coherent and D-module (as well as
derivative invariants such as homology groups) are equivalent. 

\medskip

As an application concerning the set $B\left(K\right)\backslash G\left(\mathbb{A}\right)/G\left(\mathbb{O}\right)$,
we construct a space $\bun[B\left(\t{gen}\right)]$ which is its geometrization,
and prove: 

\begin{repthm}{thm:more_cont} 

There exists an adjunction $\left(!-\t{forward},!-\t{back}\right)$
\[
\xymatrix{\dmod\left(\bun[B\left(\t{gen}\right)]\right)\ar@<5pt>[rr]^{!-\t{forward}} &  & \dmod\left(\bun\right)\ar@<5pt>[ll]^{!-\t{back}}}
\]
Moreover, $!-\t{back}$ is fully-faithful.

\end{repthm}

The pullback functor above should be interpreted as a geometric analog
of the maps
\[
\t{Fun}\left(B\left(K\right)\backslash B\left(\mathbb{A}\right)/B\left(\mathbb{O}\right)\right)=\xymatrix{\t{Fun}\left(B\left(K\right)\backslash G\left(\mathbb{A}\right)/G\left(\mathbb{O}\right)\right)\ar@<5pt>[rr]^{\t{integrate}} &  & \t{Fun}\left(G\left(K\right)\backslash G\left(\mathbb{A}\right)/G\left(\mathbb{O}\right)\right)\ar@<5pt>[ll]^{\t{pullback}}}
\]
which appear in relation to the Eisenstein series and constant term
operators, in the classical setting.

\subsection{Overview}

{}

In section \ref{sec:dom_approach} we present the first construction
schema, which we consider to be the conceptually fundamental one.
Unfortunately, conceptual appeal notwithstanding, this approach is
deficient in that invariants of the spaces so constructed are not
easy to describe (directly).

In sections \ref{sec:Quasi-maps} and \ref{sec:Bun_B_bar}, we describe
an approach which involves degenerations of regular data. This approach
is starts from the notion of \emph{quasi-maps} and the space $\overline{\t{Bun}_{B}}$
which were first presented by Finkelberg and Mikovic in \cite{FM},
and have received a fair amount of attention since. In particular,
the construction we present is the one used by Gaitsgory in \cite{DG-Whit}.
In broad terms, the upshot of this construction is that the spaces
in question are presented as quotients of a scheme (or Artin stack)
by a proper (and schematic) equivalence relation. 

In section \ref{sec:ran_approach} we describe an approach for parametrizing
generic data using the \emph{Ran Space}. This is the approach taken
by Gaitsgory in \cite{DG-Ran,DG-Whit}. This approach has the advantage
that certain invariants of the spaces so constructed are amenable
to computation via chiral homology methods. In particular, in \cite{DG-Cont}
Gaitsgory succeeds in computing the homology groups in certain cases.

In section \ref{sec:Cont_results} we use Gaitsgory's initial homology
computation for spaces of rational maps to obtain similar results
for additional moduli problems not discussed in \cite{DG-Cont}.

\subsection{Acknowledgments }

I would like to express my deep gratitude to Dennis Gaitsgory, who
was my doctoral advisor while writing this paper, for directing me
towards this project, encouraging me along the way, and for his comments
on the draft. In addition, Gaitsgory's contractibility result in \cite{DG-Cont}
plays a pivotal role in the results of section \ref{sec:Cont_results}.

\section{\label{sec:dom_approach}Moduli spaces of generic data}
\begin{notation}
Recall that $k$ is an algebraically closed field of characteristic
0, and that $X$ is a smooth connected and projective curve over $k$.
We denote by $\S$ the category of finite type schemes over $k$,
and by $\aff$ the category of finite type affine schemes over $k$.

By an \emph{$\infty$-category} we mean an $\left(\infty,1\right)$-category,
see appendix \ref{sec:Abstract-nonsense} for more details. We denote
by $\catinf$ the $\infty$-category of (small) $\infty$-categories.
We respectively denote by $\set$ and $\gpdinf$ the full subcategories
of sets and $\infty$-groupoids in $\catinf$. We shall occasionally
refer to a groupoid as a \emph{homotopy type} or a \emph{space}.

For a category $C$, we let $\pshv C$ denote the $\infty$-category
of presheaves i.e., functors $C^{\t{op}}\rightarrow\gpdinf$. In the
particular case when $C=\aff$ we use the term \emph{functor of points}
to refer to a presheaf in $\pshv{\aff}$. When $C$ is equipped with
a Grothendieck topology $\tau$, we denote the corresponding $\infty$-category
of sheaves by $\shv{C;\tau}{}$ (or omit $\tau$ when it is obvious
which topology is being used).

$ $If $C$ is a category which has been constructed to classify certain
data, we shall often denote an object of $C$ by listing the data
which it classifies (and we shall say that the data \emph{presents}
the object). For example, in definition \ref{def:dom} below, we use
the expression $\left(S,U_{S}\subseteq S\times X\right)$ to denote
an object of the category $\dom$; it should be clear from the context
what kind of data each term in the parenthesis refers to. When the
data is required to satisfy certain conditions, these are implicitly
assumed to hold and are not reflected in the notation.
\end{notation}

\subsection{Families of domains}

In the interest of motivating definition \ref{def:dom}, consider
the problem of constructing a moduli space of rational functions on
$X$ (i.e., generically defined maps to $\mathbb{A}^{1}$), $K_{X}:\aff^{\t{op}}\rightarrow\mathbf{Set}$.
An $S$-point of this functor, $f\in K_{X}\left(S\right)$ should
be presented by a rational function on $S\times X$. Let $U\subseteq S\times X$
be the largest open subscheme on which $f$ is defined. In order for
$K_{X}$ to be a functor, we must be able to pull back $f$ along
any map of schemes, $T\rightarrow S$. Requiring that these pullbacks
be defined amounts to the condition that for every $s\in S$, $U$
intersects the fiber $s\times X$. The following definition captures
this property of the domain of $f$:
\begin{defn}
\label{def:dom}$\vphantom{}$
\begin{enumerate}
\item A domain in $X$ is its non-empty open subscheme%
\footnote{In our case such a subscheme in a-priori dense. Generalizing this
notion to an arbitrary scheme, which might not be connected, one should
additionally impose a density condition. %
}. Let $S$ be a scheme, an \emph{$S$-family of domains} in $X$ is
an open subscheme $U_{S}\subseteq S\times X$ which is \emph{universally
dense} with respect to $S$. I.e., for every map of schemes $T\rightarrow S$,
after forming the pullback, 
\[
\xymatrix{U_{T}\ar[r]\ar[d] & U_{S}\ar[d]\\
T\times X\ar[r] & S\times X
}
\]
we have that $U_{T}\subseteq T\times X$ is also dense. It suffices
to check this condition at all closed points of $S$. I.e., we are
simply requiring that for every closed point, the open subscheme $U_{s}\subseteq X$
is a domain (in particular it's non-empty). 
\item The totality of all families domains in $X$ forms a (ordinary) category,
$\dom$: 

\begin{description}
\item [{An~object}] consists of the data $\left(S,U_{S}\subseteq S\times X\right)$
where $S\in\aff$, and $U_{S}$ is a family of domains in $X$.$ $ 
\item [{A~morphism}] $\left(S,U_{S}\subseteq S\times X\right)\rightarrow\left(T,U_{T}\subseteq T\times X\right)$
consists of the data of a map of affine schemes $S\xrightarrow{f}T$
which induces a commutative diagram 
\[
\xymatrix{U_{S}\ar[r]\ar[d] & U_{T}\ar[d]\\
S\times X\ar[r]^{f\times id_{X}} & T\times X
}
\]

\end{description}
\end{enumerate}
\end{defn}
There exists an evident functor 
\[
\xyR{.5pc}\xymatrix{\dom\ar[r]^{q} & \aff\\
\left(S,U_{S}\subseteq S\times X\right)\ar@{|->}[r] & S
}
\xyR{2pc}
\]
 which is a Cartesian fibration, and whose  fiber, $\left(\dom\right)_{S}$,
is the poset of $S$-families of domains in $X$ (a full subcategory
of all open sub-schemes of $S\times X$).

\subsection{Abstract moduli spaces of generic data}
\begin{notation}
A functor $C\xrightarrow{f}D$ induces, via pre-composition, a pull
back functor $\mathcal{P}shv\left(C\right)\xleftarrow{f_{*}}\mathcal{P}shv\left(D\right)$.
$f_{*}$ fits into an adjoint triple $\left(\ \lke f,f_{*},\ \rke f\right)$,
where $\ \lke f$ and $\ \rke f$ are defined on the objects of $\mathcal{P}shv\left(D\right)$,
by left and right Kan extensions (respectively) along 
\[
\xymatrix{C^{\t{op}}\ar[d]_{f}\ar[r]^{\mathcal{F}} & \gpdinf\\
D^{\t{op}}\ar@{-->}[ur]
}
\]

\end{notation}
The following definition formalizes what we mean by a moduli problem
of generic data over $X$ :
\begin{defn}
\label{dom moduli problem}The category of (abstract)\emph{ moduli
problems of generic data }is $\pshv{\dom}$. Given a presheaf $\left(\dom\right)^{\t{op}}\xrightarrow{\mathcal{F}}\gpdinf$,
its \emph{associated} functor of points is $\ \lke q\mathcal{F}\in\pshv{\aff}$,
where $q$ denotes the fibration $\dom\xrightarrow{q}\aff$.\end{defn}
\begin{example}
\label{exam:genmaps}Rational functions form a moduli problem of generic
data if we set $K_{X,\dom}:\left(\dom\right)^{\t{op}}\rightarrow\mbox{Sets}$
to be the functor 
\[
\left(S,U_{S}\subseteq S\times X\right)\mapsto\t{Hom}_{\S}\left(U_{S},\mathbb{A}^{1}\right)
\]
Its associated functor of points, $K_{X}:=\lke q\left(K_{X,\dom}\right):\affop\rightarrow\mathbf{Set}$,
sends 
\[
S\mapsto\{f\in K\left(S\times X\right):\mbox{ the domain of }f\mbox{ is an }S-\mbox{familiy of domains in }X\}
\]
where $K\left(S\times X\right)$ is the algebra of global section
of the sheaf of total quotients. I.e., the functor $K_{X}$ forgets
the evidence that $f$ is defined on a large enough domain.

Replacing $\mathbb{A}^{1}$ with an arbitrary target scheme $Y$ we
obtain similarly constructed presheaves classifying generically defined
maps from $X$ to $Y$: 
\[
\gmap Y_{\dom}:\dom^{\t{op}}\rightarrow\mathbf{Set}
\]
and its associated functor of points, which we denote 
\[
\gmap Y:\aff^{\t{op}}\rightarrow\mathbf{Set}
\]
\end{example}
\begin{rem}
\label{rem:families of generic maps}Let $\t{Map}\left(X,Y\right)$
denote the functor of points which parametrizes families of regular
maps. Since $X$ is a curve, when $Y$ is projective every generically
defined map from $X$ to $Y$ admits a (unique) extension to a regular
map defined across all of $X$. However, this is no longer the case
in families, and consequently the map $\t{Map}\left(X,Y\right)\rightarrow\gmap Y$
is not an equivalence, despite inducing an isomorphism on the set
of $k$-points. E.g., when $Y=\mathbb{P}^{1}$ the functor $\t{Map}\left(X,\mathbb{P}^{1}\right)$
has infinitely many components (labeled the degree of the map), but
$\gmap{\mathbb{P}^{1}}$ is connected.\end{rem}
\begin{example}
\label{exam:bungbk}Reduction spaces - Let $\bun[B\left(\dom\right)]:\left(\dom\right)^{\t{op}}\rightarrow\mathbf{Gpd}$
be the functor which sends $\left(S,U_{S}\subseteq S\times X\right)$
to the groupoid which classifies, up to isomorphism, the data 
\[
\left(\mathcal{P}_{G},\mathcal{P}_{B}^{U_{S}},\mathcal{P}_{B}^{U_{S}}\times_{B}G\xrightarrow[\cong]{\phi}\mathcal{P}_{G}\big|_{U_{S}}\right)
\]
where $\mathcal{P}_{G}$ is a $G$-torsor over $S\times X$, $\mathcal{P}_{B}^{U_{S}}$
is a $B$-torsor on $U_{S}$, and $\phi$ is and isomorphism of $G$-bundles
over $U_{S}$ (the data of a reduction of the structure group of $\mathcal{P}_{G}$
to $B$, over $U_{S}$). Denote its associated functor of points by
\[
\bun[B\left(\t{gen}\right)]:=\ \lke q\left(\bun[B\left(\dom\right)]\right)\in\pshv{\aff}
\]
The latter is a geometrization of the set $B\left(K\right)\backslash G\left(\mathbb{A}\right)/G\left(\mathbb{O}\right)$.
I.e., the isomorphism classes of the groupoid $\bun[B\left(\t{gen}\right)]\left(k\right)$
are in bijection with this set.

More generally, if $H$ is any subgroup of $G$, we define in a similar
way a functor of points $\bun[H\left(\t{gen}\right)]$, which classifies
families of $G$-bundles on $X$ with a generically defined reduction
to $H$. In particular, $\bun[1\left(\dom\right)]$ is the moduli
space of $G$-bundles with a generic trivialization (equivalently
a generic section).\end{example}
\begin{rem}
\label{associated ordinary} Given a presheaf $\left(\dom\right)^{\t{op}}\xrightarrow{\mathcal{F}}\gpdinf$,
recall that its associated functor of points is the left Kan extension
\[
\xymatrix{\left(\dom\right)^{\t{op}}\ar[r]\sp(0.65){\mathcal{F}}\ar[d]_{q^{\t{op}}} & \gpdinf\\
\affop\ar@{-->}[ur]_{\ \lke q\mathcal{F}}
}
\]
Ultimately, the object we wish to study is $\ \lke q\mathcal{F}$
and its invariants; $\mathcal{F}$ itself is no more than a presentation
of the former.

Noting that $q^{\t{op}}$ is a co-Cartesian fibration, it follows
that for every $S\in\aff$, we have that 
\[
\ \lke q\mathcal{F}\left(S\right)\cong\t{colim}\left(\left(\dom\right)_{S}\xrightarrow{\mathcal{F}}\gpdinf\right)
\]
I.e. the passage from $\mathcal{F}$ to $\ \lke q\mathcal{F}$ simply
identifies generic data which agrees on a smaller domain. Noting further
that $\left(\dom\right)_{S}$ is weakly contractible (it is a filtered
poset), we see that for any $\mathcal{G}\in\shv{\aff}{}$ the transformation
$\left(\ \lke q\circ q_{*}\right)\mathcal{G}\xrightarrow{\cong}\mathcal{G}$
is an equivalence, i.e. the functor $\pshv{\aff}\xrightarrow{q_{*}}\pshv{\dom}$
is fully faithful, and that $\ \lke q$ is a localization. 

We note for later use that $\ \lke q$ is left exact %
\footnote{Preserves finite limits.%
}; this follows from \cite[lemma 2.4.7]{V} after noting that $\dom$
admits all finite limits, and that the functor and $q$ preserves
these. 
\end{rem}

\subsection{D-modules\label{stable observables}}

For the most part the Langlands program is not as much interested
in the set $B\left(K\right)\backslash G\left(\mathbb{A}\right)/G\left(\mathbb{O}\right)$,
as it is in the space of functions on this set. The appropriate geometric
counterpart of this space of functions should be an appropriate category
of sheaves on the space chosen as the geometrization of the set. When
$k$ is of characteristic 0, this category is expected to be the category
of sheaves of D-modules. We now explain how to assign to every moduli
problems of generic data (presheaf on $\dom$) a category of sheaves
of D-modules. 

We consider the totality of D-modules to have the structure of a \emph{stable}
$\infty$-category, instead of a triangulated category (as is more
common); the latter is the homotopy category of the former. For a
short summary with references to notions, notation and ideas in this
below see subsection \ref{sub:Stable-observables}. We denote by $\catexl$
the $\infty$-category of stable $\infty$-categories which are co-complete,
with functors which are colimit preserving (equivalently are left
adjoints). The homotopy category of such a stable infinity category
is a triangulated category; limits in $\catexl$ should be thought
of as homotopy limits of triangulated categories.

There exists a functor $\aff^{\t{op}}\xrightarrow{\dmod}\catexl$
which assigns to a scheme $S$ (thought of as a presheaf via the Yoneda
embedding) the a stable category $\dmod\left(S\right)$, whose homotopy
category is the usual triangulated category of sheaves of D-modules
on $S$, see \ref{sub:Defining dmod} for more details. We think of
$\dmod$ as an invariant defined on schemes (whose values are categories). 

We extend $\dmod$ to arbitrary functors of points via right Kan extension
along the Yoneda embedding.
\[
\xymatrix{\affop\ar[d]^{j}\ar[r]^{\dmod} & \catexl\\
\mathcal{P}shv\left(\aff\right)^{\t{op}}\ar@{-->}[ur]_{\,\,\,\dmod:=\rke{}}
}
\]
So defined, $\dmod$ preserves small limits (\cite[lemma 5.1.5.5]{HT}). 

We further extend $\dmod$ to $\mathcal{P}shv\left(\dom\right)$ by
composing 
\[
\mathcal{P}shv\left(\dom\right)^{\t{op}}\xrightarrow{\ \lke q}\mathcal{P}shv\left(\aff\right)^{\t{op}}\xrightarrow{\dmod}\catexl
\]

\begin{rem}
\label{limit presentation} For $\mathcal{F}\in\pshv{\dom}$, the
category $\dmod\left(\mathcal{F}\right)$ can be presented as a limit
over the category $\left(\left(\dom\right)_{/\mathcal{F}}\right)^{\t{op}}$
\[
\dmod\left(\mathcal{F}\right):=\dmod\left(\ \lke q\mathcal{F}\right)\cong\lim_{\left(S,U\right)\rightarrow\mathcal{F}}\left(\dmod\left(S\right)\right)
\]
The premise of this paper is that a presheaf on $\dom$ is the conceptually
natural way of classifying structures generically defined over $X$.
However, in practice the limit presentation of $\dmod\left(\ \lke q\mathcal{F}\right)$
we obtain as above is unwieldy. In the following sections we shall
construct better, more economical, presentations of this invariant. 
\end{rem}

\subsection{Grothendieck topologies on $\dom$\label{the topology on dom}.}

In \cite[Cor 3.1.4]{GaRo} it is proven that D-modules may be descended
along fppf covers i.e., $\dmod$ factors though fppf sheafification.
For a moduli problem $\mathcal{F}\in\pshv{\dom}$, this ``continuity''
property with respect to the fppf topology can be harnessed to obtain
more economical presentations for the category $\dmod\left(\mathcal{F}\right)$,
than the one given in remark \ref{limit presentation}. To this end
we proceed to define a few natural Grothendieck topologies on $\dom$.

Let $\tau$ be the be either the Zariski, etale or fppf Grothendieck
topology on $\S$. We endow $\dom$ with a corresponding Grothendieck
pulled back from $\S$ using the functor 
\[
\xyR{.5pc}\xymatrix{\dom\ar[r] & \S\\
\left(S,U_{S}\subseteq S\times X\right)\ar@{|->}[r] & U_{S}
}
\xyR{2pc}
\]
Explicitly, a collection of morphisms in $\dom$, $\{\left(S_{i},U_{S_{i}}\subseteq S_{i}\times X\right)\rightarrow\left(S,U_{S}\subseteq S\times X\right)\}$,
is a $\tau$-cover in $\dom$ iff the collection of morphisms $\{U_{S_{i}}\rightarrow U_{S}\}$
is an $\tau$-cover in $\S$. 

Observe that for all the choices of $\tau$ above, the functor $\dom\xrightarrow{q}\aff$
is continuous in the sense that every for cover $\{\left(S_{i},U_{S_{i}}\subseteq S_{i}\times X\right)\rightarrow\left(S,U_{S}\subseteq S\times X\right)\}$
in $\dom$, its image in $\aff$, $\{S_{i}\rightarrow S\}$, is a
cover (this follows from the observation that $U_{S}\rightarrow S$
is a cover in all our topologies). Furthermore, $\dom$ and $\aff$
both admit all finite limits, and $q$ preserves these. By lemma \cite[Lemma 2.4.7]{V}
$q$ induces an adjoint pair of functors, $ $$\shv{\dom}{}\xleftarrow{\left(q^{*},q_{*}\right)}\shv{\aff}{}$
in which the functor $q_{*}$ is pullback along $q$, and the functor
$q^{*}$ is the composition of $\ \lke q$ followed by sheafification
and has the property of preserving finite limits%
\footnote{I.e., a geometric morphism of topoi.%
}. 

The upshot is, that because $\affop\xrightarrow{\dmod}\catexl$ satisfies
fppf descent, its extension $\pshv{\dom}^{\t{op}}\xrightarrow{\dmod}\catexl$
factors through $\shv{\dom;\t{fppf}}{}$. In particular, any map of
presheaves on $\dom$ which induces an equivalence after sheafification,
induces an equivalence D-module categories.

\section{\label{sec:Quasi-maps}Quasi-maps}

Recall example \textbf{\ref{exam:genmaps}}, in which we constructed
a moduli problem $\gmap Y$, classifying generically defined maps
from $X$ to $Y$. In this section we present another approach to
constructing a functor of points for this moduli problem using the
notion of \emph{quasi-maps.} The latter notion is originally due to
Drinfeld, was first described by Finkelberg and Mirkovic in \cite{FM},
and has received a fair amount of attention since. This approach has
the advantage of presenting the space of generic maps as a quotient
of a scheme by a proper (and schematic) equivalence relation.

The main result of this section is to prove that, up to sheafification
in the Zariski topology, both approaches give equivalent functors
of points. Consequently, the associated categories of D-modules are
equivalent.

For the duration of this section fix $Y\hookrightarrow\mathbb{P}^{n}$,
a scheme $Y$ \emph{together with} the data of a quasi-projective
embedding. The space of generic maps constructed using quasi maps
a-priori might depend on this embedding, however it follows from the
equivalence with $\gmap Y$ which we prove that, in fact, it does
not (up to sheafification).

\subsection{Definitions }

First, a minor matter of terminology. Let $\mathcal{V}$ and $\mathcal{W}$
be vector bundles, we distinguish between two properties of a map
of quasi-coherent sheaves $\mathcal{V}\rightarrow\mathcal{W}$: The
map is called a \emph{sub-sheaf embedding} if it is an injective map
of quasi-coherent sheaves. It is called a \emph{sub-bundle embedding}
if it is an injective map of sheaves whose co-kernel is flat (i.e.,
also a vector bundle). The latter corresponds to the notion of a map
between geometric vector bundles which is (fiber-wise) injective. 

\medskip

We start by defining the notion of a \emph{quasi map} from $X$ to
$\mathbb{P}^{n}$. Recall that that a regular map $X\rightarrow\mathbb{P}^{n}$
may be presented by the data of a line bundle $\mathcal{L}$ on $X$
together with a sub-bundle embedding $\mathcal{L}\xrightarrow{\subseteq}\mathcal{O}_{X}^{n+1}$.
A quasi-map from $X$ to $\mathbb{P}^{n}$ is a degeneration of a
regular map consisting of the data of a line bundle $\mathcal{L}$
on $X$, together with a sub-sheaf embedding $\mathcal{L}\hookrightarrow\mathcal{O}_{X}^{n+1}$
(i.e., it may not be a sub-bundle). Observe that to any quasi-map
we may associate the open subscheme $U\subseteq X$ over which $\mathcal{L}\big|_{U}\hookrightarrow\mathcal{O}_{U}^{n+1}$
is a sub-bundle - where it induces a regular map $U\rightarrow\mathbb{P}^{n}$.
In particular, to every quasi-map we may associate a generically defined
map from $X$ to $\mathbb{P}^{n}$%
\footnote{Such a map of course admits an extension to a regular map (in terms
of bundles, every invertible sub-bundle $\mathcal{L}\xrightarrow{\phi}\mathcal{O}_{X}^{n+1}$
extends to a line sub-bundle $\mathcal{L}\hookrightarrow\left(\t{Im}\left(\phi^{\vee}\right)\right)^{\vee}\subseteq\mathcal{O}_{X}^{n+1}$),
but see remark \textbf{\ref{rem:families of generic maps}.}%
}. 
\begin{defn}
Let $\qmap{\mathbb{P}^{n}}:\aff^{\t{op}}\rightarrow\mathbf{Set}$
be the functor of points whose $S$-points are presented by the data
$\left(\mathcal{L},\mathcal{L}\hookrightarrow\mathcal{O}_{S\times X}^{n+1}\right)$,
where $\mathcal{L}$ is a line bundle over $S\times X$, and $\mathcal{L}\hookrightarrow\mathcal{O}_{S\times X}^{n+1}$
is an injection of quasi-coherent sheaves, whose co-kernel is $S$-flat.
\end{defn}
If $Y\hookrightarrow\mathbb{P}^{n}$ is a locally closed subscheme,
then a quasi-map from $X$ to $Y$ should be given by the data of
a quasi-map from $X$ to $\mathbb{P}^{n}$, with the additional property
that the generic point of $X$ maps to $Y$. We proceed to define
this notion in a way better suited for families. 

In the case when $\xymatrix{Y\ar@{^{(}->}[r] & \mathbb{P}^{n}}
$ is a closed subscheme, it is defined by a graded ideal $I_{Y}\subseteq k[x_{0},\ldots,x_{n}]$.
A regular map $X\xrightarrow{f}\mathbb{P}^{n}$, presented by the
data of a sub-bundle $\mathcal{L}\subseteq\mathcal{O}_{X}^{n+1}$,
lands in $Y$ iff the composition 
\[
\t{Sym}_{X}\mathcal{L}^{\vee}\twoheadleftarrow\t{Sym}_{X}\mathcal{O}_{X}^{n+1}\cong\mathcal{O}_{X}\otimes k[x_{0},\ldots,x_{n}]\xleftarrow{}\mathcal{O}_{X}\otimes I_{Y}
\]
vanishes. We degenerate the sub-bundle requirement to obtain the notion
of a quasi map into $Y$:
\begin{defn}
When $Y\hookrightarrow\mathbb{P}^{n}$ is projective embedding, we
define $\qmap Y$ to be the subfunctor of $\qmap{\mathbb{P}^{n}}$
consisting of those points presented by the data $\left(\mathcal{L},\mathcal{L}\hookrightarrow\mathcal{O}_{S\times X}^{n+1}\right)$
such that the composition 
\[
\t{Sym}_{S\times X}\mathcal{L}^{\vee}\xleftarrow{}\t{Sym}_{S\times X}\mathcal{O}_{X}^{n+1}\xleftarrow{}\mathcal{O}_{S\times X}\otimes I_{Y}
\]
vanishes%
\footnote{The definition could have been given more economically, by replacing
the entire symmetric algebras with their finite dimensional subspaces
containing generators of $I_{Y}$.%
}.
\end{defn}
We emphasize that the definition of $\qmap Y$ depends on the embedding
$Y\hookrightarrow\mathbb{P}^{n}$, and not on $Y$ alone. 

The following lemma is well known (see e.g., \cite[lemma 3.3.1]{FM}):
\begin{lem}
$\qmap{\mathbb{P}^{n}}$ is representable by a scheme, which is moreover
a disjoint (infinite) union of projective schemes. 

If $Y\hookrightarrow\mathbb{P}^{n}$ is a projective embedding, then
$\qmap Y\rightarrow\qmap{\mathbb{P}^{n}}$ is a closed embedding.\qed\end{lem}
\begin{defn}
If $U\subseteq\mathbb{P}^{n}$ is an open subscheme then we define
\[
\qmap U\subseteq\qmap{\mathbb{P}^{n}}
\]
 to be the open subscheme which is the complement of $\qmap{\mathbb{P}^{n}\setminus U}$
(this is independent of the closed subscheme structure given to $\mathbb{P}^{n}\setminus U$). 

Finally, if $Y\hookrightarrow\mathbb{P}^{n}$ is an arbitrary locally
closed subscheme we define 
\[
\qmap Y=\qmap{\overline{Y}}\cap\qmap{\mathbb{P}^{n}\setminus\left(\overline{Y}\setminus Y\right)}
\]
It is a locally closed subscheme of $\qmap{\mathbb{P}^{n}}$. 

We point out that a map $S\rightarrow\qmap{\overline{Y}}$ lands in
the open subscheme $\qmap Y$ iff for every geometric point $s\in S\left(k\right)$,
the corresponding quasi-map carries the generic point of $X$ into
$Y$. 
\end{defn}

\subsubsection{{}}

In section \ref{sec:Bun_B_bar} we shall need to replace $\qmap Y$
with a relative and twisted version, $\qsect[][S]{S\times X}Y$ corresponding
to a scheme $Y$ over $S\times X$. The details are given at the end
of the section in \ref{sub:Quasi-sections}.

\subsection{Degenerate extensions of generic maps}

Recall the presheaves $\gmap Y$ and $\gmap Y_{\dom}$ introduced
in \ref{exam:genmaps}. There is an evident map 
\[
\qmap Y\rightarrow\gmap Y
\]
via which we think of every quasi map as presenting a generic map.
Namely, for every $S\in\aff$ it is given by the composition 
\[
\qmap Y\left(S\right)\xrightarrow{}\coprod_{\left(S,U\right)\in\dom}\gmap Y_{\dom}\left(S,U\right)\rightarrow\gmap Y\left(S\right)
\]
where the first map is given by sending a quasi-map presented by $\mathcal{L}\hookrightarrow\mathcal{O}_{S\times X}^{n+1}$
to the the open subscheme $U\subseteq S\times X$ where it is flat,
and the the regular map it defines on $U$. However, there is some
redundancy in the presentation of because a generic map may be presented
by several different quasi-maps. We introduce the equivalence relation
$\mathcal{E}_{Y}\subseteq\qmap Y\times\qmap Y$ to be the subfunctor
whose $S$-points are presented by those pairs 
\[
\left(\Big(\mathcal{L},\mathcal{L}\hookrightarrow\mathcal{O}_{S\times X}^{n+1}\Big),\left(\mathcal{L}^{'},\mathcal{L}^{'}\hookrightarrow\mathcal{O}_{S\times X}^{n+1}\right)\right)\in\left(\qmap Y\times\qmap Y\right)\left(S\right)
\]
which agree over the intersection of their regularity domains. Observe
that the following square is Cartesian 

\[
\xymatrix{\mathcal{E}_{Y}\ar[r]\ar[d]\ar@{}[dr]\sb(0.15){\ulcorner} & \qmap Y\ar[d]\\
\qmap Y\ar[r] & \gmap Y
}
\]

\bigskip

\noindent The following lemma is well known, we add a proof for completeness.
\begin{lem}
\label{lem:equiv rel}The equivalence relation $\mathcal{E}_{Y}\rightarrow\qmap Y\times\qmap Y$
is (representable by) a closed subscheme. Both the projections $\mathcal{E}_{Y}\rightarrow\qmap Y$
are proper.\end{lem}
\begin{proof}
Let us first consider the case $Y=\mathbb{P}^{n}$, and show that
the subfunctor 
\[
\mathcal{E}_{\mathbb{P}^{n}}\rightarrow\qmap{\mathbb{P}^{n}}\times\qmap{\mathbb{P}^{n}}
\]
is a closed embedding. 

We start by examining when two quasi-maps $S\xrightarrow{\phi,\psi}\qmap{\mathbb{P}^{n}}$
are generically equivalent, i.e., map to the same $S$-point of $\gmap{\mathbb{P}^{n}}$.
Let $\phi$ and $\psi$ be presented by invertible sub-sheaves 
\[
\xymatrix{\mathcal{L}_{\phi}\ar@{^{(}->}[r]^{\kappa_{\phi}} & \mathcal{O}_{S\times X}^{n+1}}
\,\,\,\t{and}\,\,\,\xymatrix{\mathcal{L}_{\psi}\ar@{^{(}->}[r]^{\kappa_{\psi}} & \mathcal{O}_{S\times X}^{n+1}}
\]
 whose co-kernels are $S$-flat. Let $U\subseteq S\times X$ be the
open subscheme where $\mathcal{L}_{\phi}\Big|_{U}\hookrightarrow\mathcal{O}_{U}^{n+1}$
is sub-bundle, and thus a maximal invertible sub-bundle. The points
$\phi$ and $\psi$ are generically equivalent iff $\mathcal{L}_{\psi}\Big|_{U}$
is a subsheaf of $\mathcal{L}_{\phi}\Big|_{U}$ (both viewed as subsheaves
of $\mathcal{O}_{U}^{n+1}$). 

Fix a vector bundle, $\mathcal{M}$, on $S\times X$ whose dual surjects
on the kernel as indicated below 
\[
\mathcal{L}_{\phi}^{\vee}\xleftarrow{\kappa_{\phi}^{\vee}}\left(\mathcal{O}_{S\times X}^{n+1}\right)^{\vee}\hookleftarrow\t{Ker}\left(\kappa_{\phi}^{\vee}\right)\twoheadleftarrow\mathcal{M}^{\vee}
\]
Dualizing and restricting to $U$ we have 
\[
\xymatrix{ & \mathcal{L}_{\psi}\Big|_{U}\ar[d]\\
\mathcal{L}_{\phi}\Big|_{U}\ar[r]_{\kappa_{\phi}}^{\subseteq} & \mathcal{O}_{U_{0}}^{n+1}\ar@{->>}[r] & \t{coker}\left(\kappa_{\phi}\right)\ar[r]^{\subseteq} & \mathcal{M}\Big|_{U}
}
\]
where map $\t{coker}\left(\kappa\right)\rightarrow\mathcal{M}$ is
injective (in fact a sub-bundle). Thus, $\phi$ and $\psi$ are generically
equivalent iff the composition $\mathcal{L}_{\psi}\Big|_{U}\rightarrow\mathcal{M}\Big|_{U}$
vanishes on $U$ iff $\mathcal{L}_{\psi}\rightarrow\mathcal{M}$ vanishes
on all of $S\times X$ (since $U\subseteq S\times X$ is dense, and
both sheaves are vector bundles).

For an arbitrary quasi-projective scheme, $Y\hookrightarrow\mathbb{P}^{n}$,
the lemma now follows from the Cartesianity of the squares below,
using the fact that both right vertical maps are proper 
\[
\xymatrix{\mathcal{E}_{Y}\ar[r]\ar[d] & \mathcal{E}_{\mathbb{P}^{n}}\ar[d]\\
\qmap Y\times\qmap{\mathbb{P}^{n}}\ar[d]\ar[r] & \qmap{\mathbb{P}^{n}}\times\qmap{\mathbb{P}^{n}}\ar[d]^{\pi_{1}}\\
\qmap Y\ar[r] & \qmap{\mathbb{P}^{n}}
}
\]

\end{proof}
\bigskip

\noindent Denote the evident simplicial object in $\pshv{\aff}$ 

\begin{equation}
\xymatrix{\cdots\ar@<2ex>[r]\ar@<-2ex>[r]\ar@{}@<0.5ex>[r]|(0.3){\fixedvdots} & \mathcal{E}_{Y}\times_{\qmap Y}\mathcal{E}_{Y}\ar[r]\ar@<2ex>[r]\ar@<-2ex>[r] & \mathcal{E}_{Y}\ar@<1ex>[r]\ar@<-1ex>[r]\ar@<1ex>[l]\ar@<-1ex>[l] & \qmap Y\ar[l]}
\label{eq:simplicial presentation}
\end{equation}
by 

\[
\xyR{0.5pc}\xymatrix{\Delta^{\t{op}}\ar[r]\sp(0.35){\mathcal{E}_{Y}^{\bullet}} & \pshv{\aff}\\
{}[n]\ar@{|->}[r] & \mathcal{E}_{Y}^{\left(n\right)}
}
\xyR{2pc}
\]
where 
\[
\mathcal{E}_{Y}^{\left(n\right)}:=\underbrace{\mathcal{E}_{Y}\times_{\qmap Y}\cdots\times_{\qmap Y}\mathcal{E}_{Y}}_{n-\t{times}}
\]

\smallskip

We denote by $\qmap Y/\mathcal{E}_{Y}$ the functor of points which
is the quotient by this equivalence relation - the colimit of this
simplicial object. However, in this case it simply reduces to the
naive pointwise quotient of sets 
\[
\left(\qmap Y/\mathcal{E}_{Y}\right)\left(S\right)=\qmap Y\left(S\right)/\mathcal{E}_{Y}\left(S\right)
\]
because $\mathcal{E}_{Y}\left(S\right)\subseteq\left(\qmap Y\left(S\right)\right)^{\times2}$
is an equivalence relation in sets.

\medskip

The functor of points $\qmap Y/\mathcal{E}_{Y}$ presents another
candidate for the ``space of generic maps'', a-priori different
from $\gmap Y$. Relative to $\gmap Y$, it has the advantage of being
concisely presented as the quotient of a scheme by a proper (schematic)
equivalence relation. The following proposition shows that both functors
are essentially equivalent, and in particular that (up to Zariski
sheafification) $\qmap Y/\mathcal{E}_{Y}$ is independent of the quasi-projective
embedding $Y\hookrightarrow\mathbb{P}^{n}$.
\begin{prop}
\label{prop:quasimap equiv}The map $\qmap Y\rightarrow\gmap Y$ induces
a map of presheaves 
\[
\qmap Y/\mathcal{E}_{Y}\xrightarrow{}\gmap Y
\]
which becomes an equivalence after Zariski sheafification.
\end{prop}
\noindent We prove this proposition, after some preparations, in
\ref{proof:prop:quasimap equiv}. First, a couple of consequences:
\begin{cor}
The Zariski sheafification of the presheaf $\qmap Y/\mathcal{E}_{Y}$
is independent of the quasi-projective embedding $Y\hookrightarrow\mathbb{P}^{n}$.
\qed
\end{cor}
The main invariant of $\gmap Y$ which we wish to study in this paper
is homology, and by extension the category of D-modules (see \ref{sub:homology of a functor}).
The following corollary is to be interpreted as providing a convenient
presentation of this category of D-modules, and using this presentation
to deduce the existence of a de-Rham cohomology functor (the left
adjoint to pullback).
\begin{cor}
\label{cor:dmod equiv gmap}~
\begin{enumerate}
\item Pullback induces an equivalence 
\[
\lim_{[n]\in\Delta^{\t{op}}}\dmod\left(\mathcal{E}_{Y}^{\left(n\right)}\right)\xleftarrow{\cong}\dmod\left(\gmap Y\right)
\]

\item Consider the pullback functors 
\[
\dmod\left(\qmap Y\right)\xleftarrow{f^{!}}\dmod\left(\gmap Y\right)\xleftarrow{t^{!}}\dmod\left(\t{spec}\left(k\right)\right)
\]
The functor $f^{!}$ always admits a left adjoint (``!-push-forward'').
When $Y\hookrightarrow\mathbb{P}^{n}$ is a closed embedding, the
functor $t^{!}$ also admits a left adjoint. 
\end{enumerate}
\end{cor}
The second assertion above is kind of ``properness'' property of
the (non-representable) map $\qmap Y\xrightarrow{}\gmap Y$, and the
functor of points $\gmap Y$ (when $Y\hookrightarrow\mathbb{P}^{n}$
is a closed embedding). 

The following remark is not used in the rest of the article. We point
it out for future use.
\begin{rem}
Corollary \ref{cor:dmod equiv gmap} implies that $\dmod\left(\gmap Y\right)$
is compactly generated. Namely, the pushforwards of compact generators
of $\dmod\left(\qmap Y\right)$ are generate because $f^{!}$ is faithful
(as is evident from (1)), and are compact because $f^{!}$ is colimit
preserving. \end{rem}
\begin{proof}
(1) is an immediate consequence of proposition \ref{prop:quasimap equiv}. 

Regarding (2), it follows from lemma \ref{lem:equiv rel} that all
the maps in $\mathcal{E}_{Y}^{\bullet}$ are proper, hence, on the
level of D-module categories, each pull-back functor admits a left
adjoint (a ``!-pushforward''). Consequently, the object assignment
\[
[n]\in\Delta\mapsto\dmod\left(\mathcal{E}_{Y}^{n}\right)\in\catexl
\]
extends to both a co-simplicial diagram (implicit in the statement
of (1)) as well as a simplicial diagram. In the former, which we denote
\[
\dmod^{!}\left(\mathcal{E}_{Y}^{\bullet}\right):\Delta\rightarrow\catexl
\]
functors are given by pullback. In the latter, which we denote 
\[
\dmod_{!}\left(\mathcal{E}_{Y}^{\bullet}\right):\Delta^{\t{op}}\rightarrow\catexl
\]
the functors are given by the left adjoints to pullback (!-pushforward).
When $Y\hookrightarrow\mathbb{P}^{n}$ is a closed embedding, each
of the $\mathcal{E}_{Y}^{\left(n\right)}$'s is proper, hence the
pushforward diagram is augmented over $\dmod\left(\spec\left(k\right)\right)$. 

Under the equivalence of (1), the functors whose adjoints we wish
to construct are identified with 
\[
\dmod\left(\qmap Y\right)\xleftarrow{}\lim_{\Delta^{\t{op}}}\dmod^{!}\left(\mathcal{E}_{Y}^{\bullet}\right)\xleftarrow{}\dmod\left(\spec\left(k\right)\right)
\]

The setup above falls into the general framework \emph{adjoint diagrams}
which we discuss in the appendix, where it is proven (\ref{lem:adjoint diagrams})
that there exists an equivalence $\underset{\Delta}{\colim}\dmod_{!}\left(\mathcal{E}_{Y}^{\bullet}\right)\xrightarrow{\cong}\underset{\Delta^{\t{op}}}{\lim}\dmod^{!}\left(\mathcal{E}_{Y}^{\bullet}\right)$,
and that under this equivalence, the pair of natural maps 
\[
\dmod\left(\qmap Y\right)\xleftarrow{}\lim_{\Delta^{\t{op}}}\dmod^{!}\left(\mathcal{E}_{Y}^{\bullet}\right)
\]
\[
\dmod\left(\qmap Y\right)\xrightarrow{}\underset{\Delta}{\colim}\dmod_{!}\left(\mathcal{E}_{Y}^{\bullet}\right)
\]
are adjoint functors.

Likewise in the case when $Y\hookrightarrow\mathbb{P}^{n}$ is a closed
embedding we conclude that 
\[
\lim_{\Delta^{\t{op}}}\dmod^{!}\left(\mathcal{E}_{Y}^{\bullet}\right)\xleftarrow{}\dmod\left(\spec\left(k\right)\right)\text{;\,\,\,\,\,\,}\underset{\Delta}{\colim}\dmod_{!}\left(\mathcal{E}_{Y}^{\bullet}\right)\xrightarrow{}\dmod\left(\spec\left(k\right)\right)
\]
are of adjoint functors. 
\end{proof}
\noindent We proceed with the preparations for the proof of Proposition
\ref{prop:quasimap equiv}.

\subsubsection{Divisor complements\label{sub:Divisor-complements}}

Recall that an effective Cartier divisor, on a scheme $Y$, is the
data of a line bundle $\mathcal{L}$ together with an injection of
coherent sheaves $\mathcal{L}\hookrightarrow\mathcal{O}_{Y}$. The
complement of the support of $\mathcal{O}_{S\times X}/\mathcal{L}$
is an open subscheme, $U_{\mathcal{L}}\subseteq Y$. We call an open
subscheme arising in this way a \emph{divisor complement}.
\begin{lem}
\label{lem:trivializing cover}Let $\left(S,U\right)\in\dom$, and
let $\mathcal{L}_{U}$ be a line bundle on $U\subseteq S\times X$.
There exists a finite Zariski cover 
\[
\left\{ \left(S_{i},U_{i}\right)\rightarrow\left(S,U\right)\right\} _{i\in I}
\]
such that for every $i$ the open subscheme $U_{i}\subseteq S_{i}\times U_{i}$
is a divisor complement. Moreover, we can choose each $U_{i}$ so
that $\mathcal{L}_{U}\Big|_{U_{i}}$ is a trivial line bundle. \end{lem}
\begin{proof}
Since $S\times X$ is quasi-projective, the topology of its underlying
topological space is generated by divisor complements Thus, we may
cover $U$ by a finite collection of open subschemes, $\left\{ U_{i}\right\} _{i\in I}$,
which trivialize $\mathcal{L}_{U}$, and such that each $U_{i}\subseteq S\times X$
is a divisor complement. Let $S_{i}\subseteq S$ be the open subscheme
which is the image of $U_{i}\subseteq S\times X\rightarrow S$. Note
that $U_{i}$ might not be a family of domains over $S$, but that
it is over $S_{i}$, and that $\left\{ \left(S_{i},U_{i}\right)\rightarrow\left(S,U\right)\right\} _{i\in I}$
is a Zariski cover in $\dom$.\end{proof}
\begin{lem}
\label{lem:density and flatness}Let $\mathcal{V}$ be a vector bundle
over $S\times X$, let $\xymatrix{\mathcal{L}\ar@{^{(}->}[r]^{\kappa} & \mathcal{V}}
$ be an invertible subsheaf, and let $U\subseteq S\times X$ be the
open subscheme where $\kappa$ is a sub-bundle embedding. Then, the
following two conditions are equivalent:
\begin{enumerate}
\item The coherent sheaf $\mathcal{V}/\mathcal{L}$ is $S$-flat.
\item The open subscheme $U\subseteq S\times X$ is universally dense relative
to $S$. I.e., The data $\left(S,U\right)$ defines a point of $\dom$.
\end{enumerate}
\end{lem}
In particular, for an effective Cartier divisor, $\mathcal{L}\hookrightarrow\mathcal{O}_{S\times X}$,
the open subscheme $U_{\mathcal{L}}\subseteq S\times X$ determines
an $S$-point of $\dom$ iff the coherent sheaf $\mathcal{O}_{S\times X}/\mathcal{L}$
is $S$-flat. 
\begin{proof}
Let $p$ and $j$ denote the maps $U\xrightarrow{j}S\times X\xrightarrow{p}S$.
Both conditions may be tested on closed points of $S$. I.e., it suffices
to show that for every maximal sheaf of ideals $\mathcal{I}_{s}\subseteq\mathcal{O}_{S}$,
corresponding to a closed point $s\in S$, we have 
\[
\t{Tor}_{S}^{1}\left(\mathcal{V}/\mathcal{L},\mathcal{I}_{s}\right)=0\,\,\,\t{iff}\,\,\, U\times_{S}\left\{ s\right\} \neq\emptyset
\]
Indeed, $\t{Tor}_{S}^{1}\left(\mathcal{V}/\mathcal{L},\mathcal{I}_{s}\right)$
vanishes iff $\mathcal{L}\Big|_{\left\{ s\right\} \times X}\xrightarrow{\kappa\big|_{\left\{ s\right\} \times X}}\mathcal{V}\Big|_{\left\{ s\right\} \times X}$
is injective iff $\kappa\big|_{\left\{ s\right\} \times X}\neq0$
iff $U\times_{S}\left\{ s\right\} \neq\emptyset$.
\end{proof}
\noindent The following lemma contains the geometric input for the
proof of proposition \ref{prop:quasimap equiv}:
\begin{lem}
\label{lem:lbextension} Assume given:
\begin{itemize}
\item $\left(S,U\right)\in\dom$.
\item $\mathcal{V}$ a rank $m$ vector bundle over $S\times X$.
\item $\mathcal{L}_{U}$ a line bundle over $U$ together with a sub-bundle
embedding 
\[
\mathcal{L}_{U}\xrightarrow{\kappa_{U}}\mathcal{V}\big|_{U}
\]

\end{itemize}
Then, there exist
\begin{itemize}
\item A Zariski cover $\left(\tilde{S},\tilde{U}\right)\xrightarrow{p}\Big(S,U\Big)$
in $\dom$.
\item A line bundle $\mathcal{L}$ on $\tilde{S}\times X$ together with
a sub-sheaf embedding $\xymatrix{\mathcal{L}\ar@{^{(}->}[r]\sp(0.4){\kappa} & \mathcal{V}\big|_{\tilde{S}\times X}}
$ whose co-kernel is $\tilde{S}$-flat.
\item An identification $\mathcal{L}\big|_{\tilde{U}}\cong\mathcal{L}_{U}\big|_{\tilde{U}}$
which exhibits $\kappa$ as an extension of 
\[
\mathcal{L}_{U}\big|_{\tilde{U}}\xrightarrow{\kappa_{U}\big|_{\tilde{U}}}\mathcal{V}\big|_{\tilde{U}}
\]

\end{itemize}
Above we have used $\left(-\right)\big|_{\tilde{U}}$ to denote pullback
along $\tilde{U}\rightarrow U$.\end{lem}
\begin{proof}
According to lemma \ref{lem:trivializing cover}, we may find a Zariski
cover in $\dom$, $\left(\tilde{S},\tilde{U}\right)\rightarrow\left(S,U\right)$,
such that 
\begin{itemize}
\item $\mathcal{L}_{U}\Big|_{\tilde{U}}$ is a trivial line bundle.
\item The open subscheme $\tilde{U}\subseteq\tilde{S}\times X$ is a divisor
complement associated to a Cartier divisor $\mathcal{N}\hookrightarrow\mathcal{O}_{\tilde{S}\times X}$.
\end{itemize}
We proceed to show that the sub-bundle embedding 
\[
\left(*\right)\,\,\,\,\,\,\,\,\,\mathcal{L}_{U}\big|_{\tilde{U}}\xrightarrow{\kappa_{U}\big|_{\tilde{U}}}\mathcal{V}\big|_{\tilde{U}}
\]
 admits a degenerate extension across $\tilde{S}\times X$. We point
out that the line bundle $\mathcal{N}$ is trivialized over $\tilde{U}$,
and we fix identifications $\mathcal{N}\big|_{\tilde{U}}=\mathcal{O}_{\tilde{U}}\cong\mathcal{L}_{U}\big|_{\tilde{U}}$.
By a standard lemma in algebraic geometry (\cite[II.5.14]{Hart}),
there exists an integer $l$ and a map of coherent sheaves $\mathcal{N}^{\otimes l}\xrightarrow{\kappa}\mathcal{V}\big|_{\tilde{S}\times X}$
whose restriction to $\tilde{U}$ may be identified with $\left(*\right)$.
$ $By lemma \ref{lem:density and flatness}, $\t{coker}\left(\kappa\right)$
is $\tilde{S}$-flat.
\end{proof}
{}\medskip

\subsubsection{Proof of proposition \ref{prop:quasimap equiv}\label{proof:prop:quasimap equiv}}

The following square is Cartesian

\[
\xymatrix{\qmap Y\ar[r]\ar[d]\ar@{}[dr]\sb(0.15){\ulcorner} & \qmap{\mathbb{P}^{n}}\ar[d]\\
\qmap Y\ar[r] & \gmap{\mathbb{P}^{n}}
}
\]
hence it suffices to prove the proposition for $Y=\mathbb{P}^{n}$.

It suffices to fix an $S$-point, $S\rightarrow\gmap{\mathbb{P}^{n}}$,
and show that there exists a Zariski cover $\tilde{S}\rightarrow S$,
and a lift as indicated by the dotted arrow below 
\[
\xymatrix{ &  & \qmap{\mathbb{P}^{n}}\ar[d]\\
\tilde{S}\ar[r]\ar@{-->}[urr] & S\ar[r]\sp(0.3){\phi} & \gmap{\mathbb{P}^{n}}
}
\]
Let $\phi$ be presented by the data of a point $\left(S,U\right)\in\dom$,
and a sub-bundle embedding $\mathcal{L}_{U}\xrightarrow[\subseteq]{\kappa_{U}}\mathcal{O}_{U}^{n+1}$
over $U$. $ $ Lemma \ref{lem:lbextension} guarantees the existence
of a cover $\tilde{S}\rightarrow S$, and an invertible sub-sheaf
$\xymatrix{\mathcal{L}\ar@{^{(}->}[r]^{\kappa} & \mathcal{O}_{\tilde{S}\times X}^{n+1}}
$, which is an extension of $\kappa_{U}\Big|_{\tilde{U}}$ to all of
$\tilde{S}\times X$. The data associated with $\kappa$ presents
a map, $\tilde{S}\rightarrow\qmap{\mathbb{P}^{n}}$, which is the
sought after lift.

\subsection{Quasi sections\label{sub:Quasi-sections}}

In the next section we shall need a relative and twisted generalization
of the notion of quasi map, which we now define. All the results in
this section, proven above, could have been stated and proven in this
more general setup (at the cost of encumbering the presentation).

Fix $S\in\aff$, and let $\mathcal{V}$ be a vector bundle on $S\times X$.
Denote the relative projectivization by 
\[
\mathbb{P}\left(\mathcal{V}\right):=\t{proj}_{S\times X}\left(\t{sym}_{\mathcal{O}_{S\times X}}\mathcal{V}^{*}\right)
\]
it is a locally projective scheme over $S\times X$. We define the
space of quasi-sections of $\mathbb{P}\left(\mathcal{V}\right)\rightarrow S\times X$,
relative to $S$:
\begin{defn}
~
\begin{enumerate}
\item The functor 
\[
\qsect[][S]{S\times X}{\mathbb{P}\left(\mathcal{V}\right)}:\aff_{/S}^{\t{op}}\rightarrow\set
\]
is defined to be the functor of points over $S$, whose $T$-points
are presented by the data $\left(\mathcal{L},\mathcal{L}\hookrightarrow\mathcal{V}\big|_{T\times X}\right)$,
where $\mathcal{L}$ is a line bundle over $T\times X$, and $\mathcal{L}\hookrightarrow\mathcal{V}\big|_{T\times X}$
is an injection of quasi-coherent sheaves, whose co-kernel is $T$-flat.
\item For a closed embedding $Y\hookrightarrow\mathbb{P}\left(\mathcal{V}\right)$,
defined by a graded sheaf of ideals $\mathcal{I}_{Y}\subseteq\t{Sym}_{T\times X}\mathcal{V}^{\vee}$,
we define 
\[
\qsect[][S]{S\times X}Y\subseteq\qsect[][S]{S\times X}{\mathbb{P}\left(\mathcal{V}\right)}
\]
to be the subfunctor of consisting of those points presented by the
data $\left(\mathcal{L},\mathcal{L}\hookrightarrow\mathcal{V}\big|_{T\times X}\right)$
such that the composition 
\[
\t{Sym}_{T\times X}\mathcal{L}^{\vee}\xleftarrow{}\t{Sym}_{T\times X}\mathcal{V}^{\vee}\xleftarrow{}\mathcal{I}_{Y}
\]
vanishes. 
\end{enumerate}
\end{defn}
\noindent When $S=\spec\left(k\right)$, and $\mathcal{V}=\mathcal{O}_{S\times X}^{n+1}$
this definition reduces to $\qmap Y$.

As for the absolute version, there exists a map 
\[
\qsect[][S]{S\times X}Y\rightarrow\gsect[][S]{S\times X}Y
\]
and the counterpart of proposition \ref{prop:quasimap equiv} holds.
The proof is virtually identical (after adjusting notation), and is
omitted.
\begin{prop}
\label{prop:quasisect equiv}The map $\qsect[][S]{S\times X}Y\rightarrow\gsect[][S]{S\times X}Y$
induces a map of presheaves 
\[
\qsect[][S]{S\times X}Y/\mathcal{E}_{Y}\xrightarrow{}\gsect[][S]{S\times X}Y
\]
which becomes an equivalence after Zariski sheafification.\qed
\end{prop}

\section{\label{sec:Bun_B_bar}D- modules ``on'' $B\left(K\right)\backslash G\left(\mathbb{A}\right)/G\left(\mathbb{O}\right)$}

Recall example \ref{exam:bungbk}, in which we introduced a moduli
problem of generic data $\bun[B\left(\dom\right)]\in\pshv{\dom}$,
and denoted its associated functor of points by 
\[
\bun[B\left(\t{gen}\right)]:=\lke q\left(\bun[B\left(\dom\right)]\right)\in\pshv{\aff}
\]
It is a geometrization%
\footnote{``The geometrization'' according to the premise of this paper.%
} of $B\left(K\right)\backslash G\left(\mathbb{A}\right)/G\left(\mathbb{O}\right)$.
Conceptual appeal notwithstanding, this presentation of $\bun[B\left(\t{gen}\right)]$
is too unwieldy to be of much value. Namely, the issue is that using
it (directly) to obtain presentations of invariants such as $\dmod$
is a non-starter.

In the note \cite[subsection 1.1]{DG-Whit}, Gaitsgory introduces
a category denoted $\dmod\left(\t{Bun}_{B}^{\t{rat}}\right)$, which
is cast to play the role of the category of D-modules ``on'' $B\left(K\right)\backslash G\left(\mathbb{A}\right)/G\left(\mathbb{O}\right)$.
In this section we present the construction of Gatisgory's category,
and show that it is equivalent to $\dmod\left(\bun[B\left(\t{gen}\right)]\right)$.
The discussion parallels that of the previous section.
\begin{notation}
Let $G$ be a connected reductive affine algebraic group. Choose a
Borel subgroup $B$, denote by $N$ the unipotent radical of $B$,
and by $H=B/N$ the canonical Cartan. Choose a root system for $G$
and $B$, and denote by $\Lambda_{G}^{+}$ the semi-group of dominant
integral weights. For a dominant integral weight $\lambda$, let $V^{\lambda}$
denote the irreducible representation of $G$ with highest weight
$\lambda$. For a $H$-torsor, $\mathcal{P}_{H}$, we denote by $\lambda\left(\mathcal{P}_{H}\right)$
the $\mathbb{G}_{m}$-torsor $\mathcal{P}_{H}\times_{\lambda}\mathbb{G}_{m}$
(as well as the associated line bundle - a quasi-coherent sheaf).
For a $G$-torsor, $\mathcal{P}_{G}$, we denote by $\mathcal{V}_{\mathcal{P}_{G}}^{\lambda}$
the vector bundle corresponding to $V^{\lambda}$.
\end{notation}

\subsection{Constructions}

\subsubsection{Plucker data }

Given a scheme $Y$ and a $G$-bundle $\mathcal{P}_{G}$ on $Y$,
a convenient way of presenting the data of a reduction of the structure
group of $\mathcal{P}_{G}$ to $B$ is given by specifying an $H$-bundle,
$\mathcal{P}_{H}$, together with bundle maps for every $\lambda\in\Lambda_{G}^{+}$
\[
\lambda\left(\mathcal{P}_{H}\right)\xrightarrow[\subseteq]{\kappa_{\lambda}}\mathcal{V}_{\mathcal{P}_{G}}^{\lambda}
\]
 which satisfy the \emph{Plucker relations. }I.e., for $\lambda_{0}$
the trivial character, $\kappa^{0}$ is the identity map 
\[
\mathcal{O}\cong\lambda_{0}\left(\mathcal{P}_{H}\right)\rightarrow\mathcal{V}_{\mathcal{P}_{G}}^{0}\cong\mathcal{O}
\]
and for every pair of dominant integral weights the following diagram
commutes\emph{ 
\[
\xymatrix{\left(\lambda+\mu\right)\left(\mathcal{P}_{H}\right)\ar[r]\sp(0.55){\kappa_{\lambda+\mu}}\ar[d] & \mathcal{V}_{\mathcal{P}_{G}}^{\lambda+\mu}\ar[d]\\
\lambda\left(\mathcal{P}_{H}\right)\otimes\mu\left(\mathcal{P}_{H}\right)\ar[r]\sp(0.55){\kappa_{\lambda}\otimes\kappa_{\mu}} & \mathcal{V}_{\mathcal{P}_{G}}^{\mu}\otimes\mathcal{V}_{\mathcal{P}_{G}}^{\lambda}
}
\]
}From now on, we adopt this Plucker point of view for presenting points
of $\bun[B\left(\t{gen}\right)]$.

\subsubsection{Degenerate reduction spaces}

Degenerating the data of a reduction of a $G$-torsor to $B$, in
a similar fashion to the degeneration of a regular map to a quasi-map,
we obtain Drinfeld's (relative) compactification of $\t{Bun}_{B}\rightarrow\bun$:

\medskip

Let $\bunbbar{}\in\pshv{\aff}$%
\footnote{Often denoted by some variation on $\overline{\t{Bun}_{B}}$.%
}, be the presheaf which sends a scheme $S$ to the groupoid which
classifies the data 
\[
\left(\mathcal{P}_{G},\mathcal{P}_{H},\lambda\left(\mathcal{P}_{H}\right)\xrightarrow{\kappa_{\lambda}}\mathcal{V}_{\mathcal{P}_{G}}^{\lambda}:\lambda\in\Lambda_{G}^{+}\right)
\]
where:
\begin{itemize}
\item $\mathcal{P}_{G}$ is a $G$-torsor on $S\times X$ .
\item $\mathcal{P}_{H}$ is an $H$-torsor on $S\times X$ .
\item For every $\lambda\in\Lambda_{G}^{+}$, $\kappa_{\lambda}$ is an
injection of coherent sheaves whose co-kernel is $S$-flat. The collection
of $\kappa_{\lambda}$'s is required to satisfy the Plucker relations. 
\end{itemize}
\medskip

Informally, this is a moduli space of $G$-bundles on $X$, with a
degenerate reduction to $B$. There is an evident map $\t{Bun}_{B}\rightarrow\overline{\bun[B]}$
whose image consists of those points for which the $\kappa_{\lambda}$'s
are sub-bundle embeddings. For more details on $\bunbbar{}$ see \cite{FM}
or \cite{BG}. 

\medskip

Let Let $\left\{ \lambda_{j}\right\} _{j\in\mathcal{J}}\subseteq\Lambda_{G}^{+}$
be a finite subset which generates $\Lambda_{G}^{+}$ over $\mathbb{Z}_{\geq0}$.
The natural map%
\footnote{Which maps $1\in G$ to the highest weight line in each component.%
} 
\[
G/B\rightarrow\underset{j\in\mathcal{J}}{\times}\mathbb{P}\left(V^{\lambda_{j}}\right)\hookrightarrow\mathbb{P}\left(\otimes_{j\in\mathcal{J}}V^{\lambda_{j}}\right)
\]
 is a closed embedding. For every $j\in\mathcal{J}$ let $\mathcal{V}^{\lambda_{j}}$
be the vector bundle on $\bun\times X$ corresponding to the representation
$V^{\lambda_{j}}$, and let $\mathcal{V}:=\otimes_{j\in\mathcal{J}}\mathcal{V}^{\lambda_{j}}$. 
\begin{lem}
\cite[prop. 1.2.2]{BG}\label{lem:bunbbar representable} Let $S\rightarrow\bun$
classify a $G$-bundle $\mathcal{P}_{G}$ on $S\times X$, and denote
$\left(\bunbbar{}\right)_{S}:=S\times_{\bun}\bunbbar{}$. 

There exists a natural isomorphism 
\[
\left(\bunbbar{}\right)_{S}\xrightarrow{\cong}\qsect[][S]{S\times X}{\mathcal{P}_{G}/B}
\]
where the space of quasi-sections is defined via the closed embedding
\[
\mathcal{P}_{G}/B\hookrightarrow\mathbb{P}\left(\mathcal{V}\Big|_{S\times X}\right)
\]
In particular $\bunbbar{}$ is schematic and proper over $\bun$.
\qed\end{lem}
\begin{example}
When $G=SL_{2}$ the presheaf $\overline{\t{Bun}_{SL_{2}}^{B}}$ is
equivalent to the presheaf which sends a scheme $S$ to the groupoid
$\overline{\t{Bun}_{SL_{2}}^{B}}\left(S\right)$ classifying the data
$\left(\mathcal{L},\mathcal{V},\mathcal{L}\hookrightarrow\mathcal{V}\right)$,
where $\mathcal{L}$ is a line bundle on $S\times X$, $\mathcal{V}$
is a rank-$2$ vector bundle on $S\times X$ with trivial determinant,
and $\mathcal{L}\hookrightarrow\mathcal{V}$ is an injection of quasi-coherent
sheaves whose co-kernel is flat over $S$. 

Observe that when $S=\spec\left(k\right)$, we may associate to a
every degenerate reduction $\left(\mathcal{L},\mathcal{V},\mathcal{L}\hookrightarrow\mathcal{V}\right)\in\overline{\t{Bun}_{SL_{2}}^{B}}\left(k\right)$
the genuine reduction $\left(\tilde{\mathcal{L}},\mathcal{V},\tilde{\mathcal{L}}\xrightarrow{\subseteq}\mathcal{V}\right)\in\t{Bun}_{B}\left(k\right)$
where $\tilde{\mathcal{L}}$ is the maximal sub-bundle, $\mathcal{L}\hookrightarrow\tilde{\mathcal{L}}\subseteq\mathcal{V}$
extending $\mathcal{L}$. However, there may not exist such extension
for an arbitrary $S$-family%
\footnote{For essentially the same reason that a continuous function $\mathbb{R}^{2}\setminus\left\{ 0\right\} \rightarrow\mathbb{R}$
may admit a continuous extension when restricted to any path, but
nonetheless fail to admit a global continuous extension. %
}.
\end{example}
We wish to use $\overline{\bun[B]}$ to construct a geometrization
for $B\left(K\right)\backslash G\left(\mathbb{A}\right)/G\left(\mathbb{O}\right)$.
Note that on the level of $k$-points there exists a surjective map
\[
\pi_{0}\left(\overline{\bun[B]}\left(k\right)\right)\rightarrow B\left(K\right)\backslash G\left(\mathbb{A}\right)/G\left(\mathbb{O}\right)
\]
but that this map is not bijective.

\subsubsection{{}}

Gaitsgory's $\dmod\left(\t{Bun}_{B}^{\t{rat}}\right)$ of \cite[subscetion 1.1]{DG-Whit}
may be defined as follows: To every point $P\in\overline{\bun[B]}\left(S\right)$
we may associate its \emph{regular} domain $U_{P}\subseteq S\times X$,
this is the maximal open subscheme where the Plucker data is regular,
and hence defines a genuine structure reduction of $\mathcal{P}_{G}\big|_{U_{P}}$
to $B$. 

Define $\mathcal{H}\in\pshv{\aff}$ be the presheaf which sends $S$
to the groupoid classifying the data 
\[
\left(P\in\overline{\bun[B]},P'\in\overline{\bun[B]},\phi\right)
\]
 where $\phi$ is an isomorphism of the underlying $G$-torsors (defined
on all of $S\times X$), which commutes with the $\kappa_{\lambda}$'s
over $U_{P}\cap U_{P^{'}}$ (hence induces an isomorphism of $B$-reductions
there). It is evident that $\mathcal{H}$ admits a groupoid structure
(in presheaves) over $\overline{\bun[B]}$. In loc. cit., $\dmod\left(\t{Bun}_{B}^{\t{rat}}\right)$
is defined to be the category of equivariant D-modules with respect
to this groupoid.

On the level of points, we may define $\overline{\bun[B]}^{\mathcal{H}}$
to be the quotient of $\bunbbar{}$ by this groupoid (i.e., the colimit
of the associated simplicial object in $\pshv{\aff}$). It follows
that $\dmod\left(\t{Bun}_{B}^{\t{rat}}\right)\cong\dmod\left(\bunbbar{}^{\mathcal{H}}\right)$.
After taking this quotient, we do have an identification of sets 
\[
\pi_{0}\left(\bunbbar{}^{\mathcal{H}}\left(k\right)\right)\cong B\left(K\right)\backslash G\left(\mathbb{A}\right)/G\left(\mathbb{O}\right)
\]

The main result of this section is:
\begin{prop}
\label{prop:B_bar points equiv}There exists a map in $\pshv{\aff}$
\[
\left(\bunbbar{}\right)^{\mathcal{H}}\rightarrow\bun[B\left(\t{gen}\right)]
\]
which becomes an equivalence after sheafification in the Zariski topology.
\end{prop}
\medskip

\noindent The following corollary is of particular interest in the
geometric Langlands program:
\begin{cor}
\label{cor:dmod equiv}~~
\begin{enumerate}
\item Pullback along the map constructed in \ref{prop:B_bar points equiv}
gives rise to an equivalence 
\[
\lim_{[n]\in\Delta^{\t{op}}}\left(\dmod\left(\mathcal{H}^{\left(n\right)}\right)\right)\cong\dmod\left(\left(\bunbbar{}\right)^{\mathcal{H}}\right)\leftarrow\dmod\left(\bun[B\left(\t{gen}\right)]\right)
\]
where 
\[
\mathcal{H}^{\left(n\right)}:=\underbrace{\mathcal{H}\times_{\bunbbar{}}\cdots\times_{\bunbbar{}}\mathcal{H}}_{n-\t{times}}
\]

\item The pullback functors 
\[
\dmod\left(\bunbbar{}\right)\xleftarrow{}\dmod\left(\bun[B\left(\t{gen}\right)]\right)\xleftarrow{}\dmod\left(\bun\right)
\]
admit left adjoints (``!-push-forward'').
\end{enumerate}
\end{cor}
In Theorem \ref{thm:more_cont} we shall prove that the pullback functor
is moreover fully-faithful. The proof of this corollary is completely
analogous to that of corollary  \ref{cor:dmod equiv gmap}.

\subsubsection{Proof of proposition \ref{prop:B_bar points equiv}}

We proceed to reduce the statement to proposition \ref{prop:quasisect equiv}.
For every $S\rightarrow\bun$ denote 
\[
\left(\bunbbar{}\right)_{S}:=S\times_{\bun}\bunbbar{}
\]
and denote similarly for $\bun[B\left(\t{gen}\right)]$ and $\mathcal{H}$. 

It follows from lemma \ref{lem:bunbbar representable} that 
\[
\left(\bunbbar{}\right)_{S}\cong\qsect[][S]{S\times X}{\mathcal{P}_{G}^{S}/B}
\]
and it is evident that 
\[
\left(\bun[B\left(\t{gen}\right)]\right)_{S}\cong\gsect[][S]{S\times X}{\mathcal{P}_{G}^{S}/B}
\]
and that $\mathcal{H}_{S}$ is equivalent to the fiber product 
\[
\xymatrix{\mathcal{H}_{S}\ar[r]\ar[d] & \qsect[][S]{S\times X}{\mathcal{P}_{G}^{S}/B}\ar[d]\\
\qsect[][S]{S\times X}{\mathcal{P}_{G}^{S}/B}\ar[r] & \gsect[][S]{S\times X}{\mathcal{P}_{G}^{S}/B}
}
\]

Thus we obtain maps, for every $S\rightarrow\bun$, 
\[
\left(\bunbbar{\left(\t{gen}\right)}\right)_{S}\Big/\mathcal{H}_{S}\xrightarrow{\cong}\left(\bun[B\left(\t{gen}\right)]\right)_{S}
\]
which become equivalences after sheafification in the Zariski topology
by proposition \ref{prop:quasisect equiv}. These maps are all natural
in $S\rightarrow\bun$, and we conclude the existence of a map of
presheaves 
\[
\bunbbar{}^{\mathcal{H}}\cong\underset{S\rightarrow\bun}{\colim}\left(\bunbbar{}\right)_{S}\Big/\mathcal{H}_{S}\xrightarrow{}\underset{S\rightarrow\bun}{\colim}\left(\bun[B\left(\t{gen}\right)]\right)_{S}\cong\bun[B\left(\t{gen}\right)]
\]
which becomes an equivalence after sheafification in the Zariski topology.
\qed
\begin{rem}
\emph{Drinfeld's Parabolic structures}. In \cite[1.3]{BG}, Braverman
and Gaitsgory consider two different notions (attributed to Drinfeld
in loc. cit.) of a degenerate reduction of a $G$-torsor (on $X$)
to $P$. These two notions agree in the case when $P=B$, but differ
in general. Correspondingly, they construct two different relative
compactification of the map $\t{Bun}_{P}\rightarrow\bun$, denoted
$\overline{\t{Bun}_{P}}$ and $\widetilde{\t{Bun}_{P}}$, both schematic
and proper over $\bun$. The categories of D-modules $\dmod\left(\overline{\t{Bun}_{P}}\right)$
and $\dmod\left(\widetilde{\t{Bun}_{P}}\right)$ have received a fair
amount of attention (e.g., in \cite{BG,BFGM}) due to their part in
the construction of a geometric ``Eisenstein series'' functor 
\[
\dmod\left(\bun\right)\xleftarrow{\t{Eis}_{M}^{G}}\dmod\left(\t{Bun}_{M}\right)
\]
where $M$ is the Levi factor of $P$.

It can be shown that $\overline{\t{Bun}_{P}}$ and $\widetilde{\t{Bun}_{P}}$
give rise to  two different presentations of $\bun[P\left(\t{gen}\right)]$
(up to fppf sheafification) as a quotient of a scheme (relative to
$\bun$) by a schematic and proper equivalence relation 
\[
\]
\[
\left(\overline{\t{Bun}_{P}}\right)/\overline{\mathcal{H}_{P}}\xrightarrow{}\bun[P\left(\t{gen}\right)]\,\,\,\t{and}\,\,\,\left(\widetilde{\t{Bun}_{P}}\right)/\widetilde{\mathcal{H}_{P}}\xrightarrow{}\bun[P\left(\t{gen}\right)]
\]
Consequently, we obtain two different presentations for the category
of D-modules on $\bun[P\left(\t{gen}\right)]$ as a category of equivariant
objects $ $ $ $
\[
\dmod\left(\left(\overline{\t{Bun}_{P}}\right)\right)^{\overline{\mathcal{H}_{P}}}\xleftarrow{\cong}\dmod\left(\bun[P\left(\t{gen}\right)]\right)
\]
and 
\[
\dmod\left(\left(\widetilde{\t{Bun}_{P}}\right)\right)^{\widetilde{\mathcal{H}_{P}}}\xleftarrow{\cong}\dmod\left(\bun[P\left(\t{gen}\right)]\right)
\]
\end{rem}

\specialsection{\label{sec:ran_approach}The Ran Space Approach to Parametrizing
Domains}

In this section we describe an approach to presenting moduli problems
of generic data using\emph{ }presheaves over the \emph{Ran space}.
This approach has the advantage that D-modules presented this way
are amenable to Chiral Homology techniques. After presenting the framework
we prove that it is equivalent, in an appropriate sense, to the $\dom$
approach.

\subsection{The Ran Space}

Let $\finsur$ denote the category of finite sets with surjections
as morphisms. The \emph{Ran space, }denoted $\ran$, is the colimit
of the diagram 
\[
\finsurop\xrightarrow{I\mapsto X^{I}}\pshv{\aff}
\]
in which a surjection of finite sets $J\twoheadleftarrow I$ maps
to the corresponding diagonal embedding $X^{J}\hookrightarrow X^{I}$.
In the appendix (\ref{lem:ran fibd in sets}) it is proven that  a
point $S\rightarrow\ran$ is equivalent to the data of a finite subset
$F\subset\t{Hom}\left(S,X\right)$, i.e., $\ran\left(S\right)\in\gpd_{\infty}$
is the set of finite subsets of $\t{Hom}\left(S,X\right)$. Note that
$\ran$ is not a sheaf even in the Zariski topology (it is not separated),
and in any case its sheafifications are not representable. by a scheme
or ind-scheme.

Morally, the Ran space should be thought of as the moduli space for
finite subsets%
\footnote{We emphasize the distinction between finite subsets and finite subschemes.%
} of $X$, as is reflected in the fact that a closed point $\spec\left(k\right)\rightarrow\ran$
corresponds to a finite subset $F\subset X\left(k\right)$. More generally,
to a point $S\rightarrow\ran$ classified by $F=\left\{ f_{1},\ldots,f_{n}\right\} \subseteq\t{Hom}\left(S,X\right)$,
we associate the closed subspace $\Gamma_{F}:=\cup\Gamma_{f_{i}}\subseteq S\times X$,
where $\Gamma_{f_{i}}\subseteq S\times X$ is the graph of $S\xrightarrow{f_{i}}X$.
However, since we are concerned with generic data, we take the opposite
perspective and interpret such a point as parametrizing the complement
open subscheme 
\[
U_{F}:=\left(S\times X\right)\setminus\Gamma_{F}
\]
which is family of domains in the sense of \ref{def:dom}. We point
out that because $X$ is a curve, every open subscheme is the complement
of a finite collection of points, whence we are justified in thinking
of $\ran$ as a moduli of open subschemes of $X$. It would seem that
for a higher dimensional scheme in place of $X$, this approach would
not be reasonable. 

There are two differences between $\ran$ and $\dom$. The first concerning
objects, is that not every family of domains $\left(S,U\right)\in\dom$
may be presented using a map $S\rightarrow\ran$, thus $\ran$ classifies
a restrictive collection of domain families - \emph{graph complements}.
The second difference, concerning morphisms, is that while the fibers
$\dom$ over $\aff$ are posets, $\ran$ takes values in sets, and
thus does not account for the inclusion of one finite subset in another.

\subsubsection{Preview\label{sub:monad Motivation}\label{sub:Preview}}

Consider the moduli problem of classifying generically defined maps
from $X$ to $Y$. Construct a presheaf 
\[
\gmap Y_{\ran}\in\pshv{\aff}
\]
by defining that a map $S\rightarrow\gmap Y_{\ran}$ is presented
by the data 
\[
\left(S\xrightarrow{F}\ran,U_{F}\xrightarrow{f}Y\right)
\]
where $f$ is regular map. I.e., there is a map of presheaves 
\[
\gmap Y_{\ran}\rightarrow\ran
\]
and the points of $\gmap Y_{\ran}$, lying over a point $S\xrightarrow{F}\ran$,
classify those generic maps from $S\times X$ to $Y$ which are regular
on $U_{F}$. We would like to endow the functor of points $\gmap Y_{\ran}$
with additional structure which reflects the fact that a pair of its
points may be parametrizing the same generic data, but with different
domain data. 

\medskip

In the next subsection we shall define the Ran version of the category
of \emph{moduli problems of generic data} over $X$, and compare it
to the $\dom$ version defined in \ref{dom moduli problem} - they
will be almost equivalent. Namely, the respective functors of points
in $\pshv{\aff}$ associated to each formulation will be proven to
be equivalent after fppf sheafification. 

This Ran version will be defined as the category of modules for monad
acting on $\pshv{\aff}_{/\ran}$. This formulation, using the Ran
space and the monad, is more economical than the one via presheaves
on $\dom$, and gives rise to presentations of various invariants
of the moduli problem which are more approachable. 

A first approximation to the construction of the Ran version is as
follows: Let 
\[
\Ran\rightarrow\aff
\]
be the Cartesian fibration which is the Grothendieck un-straightening
of the functor 
\[
\aff^{\t{op}}\xrightarrow{\ran}\set
\]
I.e., $\Ran$ is the category of points of the functor of points $\ran$.
There is an evident functor $i'$ making the following diagram commute
\[
\xymatrix{\Ran\ar[rr]^{i'}\ar[dr] &  & \dom\ar[dl]^{q}\\
 & \aff
}
\]
The monad which we will use is a more economical version of the monad
on $\pshv{\aff}_{/\ran}\cong\pshv{\Ran}$ induced by the adjunction
\[
\xymatrix{\pshv{\Ran}\ar@<1ex>[r]^{\lke{i'}} & \pshv{\dom}\ar@<1ex>[l]^{i'_{*}}}
\]
The monad we shall use will be constructed using an intermediate domain
category presented in \ref{sub:domgamma}.

It turns out that for $\mathcal{F}\in\pshv{\dom}$, invariants such
as its homology are equivalent to those of $i_{*}^{'}\mathcal{F}\in\pshv{\Ran}\cong\pshv{\aff}_{/\ran}$,
which are often computable. In particular, in the example of generic
maps, Gaitsgory computes the homology of $\gmap Y_{\ran}$ (for certain
choices of $Y$), from which we deduce the homology of $\gmap Y$
(this is the topic of section \ref{sec:Cont_results}).

\subsection{$\dom^{\Gamma}$ - A more economical category of domains\label{sub:domgamma}}

Recall the Cartesian fibration $\Ran\rightarrow\aff$ defined in \ref{sub:Preview}.
As mentioned above, an object of $\Ran$ lying over a scheme $S$
may be interpreted as presenting a family of domains in $S\times X$
(the \emph{graph complement}), however $\Ran$ doesn't include morphisms
which account for the inclusion of one domain in another. We construct
the category $\dom^{\Gamma}$ by adding the appropriate morphisms:

\subsubsection{Construction}

Let $\dom^{\Gamma}$ be the following category: 
\begin{description}
\item [{An~object}] consists of the data $\left(S,F\subseteq\t{Hom}\left(S,X\right)\right)$,
where $S$ is an affine scheme, and $F$ is a non-empty finite subset. 
\item [{A~morphism}] $\left(S,F\subseteq\t{Hom}\left(S,X\right)\right)\rightarrow\left(T,G\subseteq\t{Hom}\left(T,X\right)\right)$
is a map of schemes $S\xrightarrow{f}T$ such that pre-composition
with $f$ carries $G$ into $F$. 
\end{description}
\medskip

It is evident that $\dom^{\Gamma}$ is sandwiched in a commuting diagram
\begin{equation}
\xymatrix{\Ran\ar[r]^{i}\ar[drr]^{s} & \dom^{\Gamma}\ar[rr]^{p}\ar[dr]^{r} &  & \dom\ar[dl]_{q}\\
 &  & \aff
}
\label{eq:functors}
\end{equation}
 in which $p$ associates to each $\left(S,F\subseteq\t{Hom}\left(S,X\right)\right)\in\dom^{\Gamma}$
the family of domains $U_{F}:=S\times X\setminus\Gamma_{F}$, where
$\Gamma_{F}$ is the union of the graphs of the maps in $F\subseteq\t{Hom}\left(S,X\right)$.
Note that all three diagonal maps are Cartesian fibrations, and that
$i$ and $p$ preserve morphisms which are Cartesian over $\aff$.
We remark that $p$ is not full (but it is faithful). 

We endow $\dom^{\Gamma}$ with the fppf Grothendieck topology pulled
back from $\dom$ along $p$. I.e., a collection of morphisms 
\[
\left\{ \left(S_{i},F_{i}\subseteq\t{Hom}\left(S_{i},X\right)\right)\rightarrow\left(S,F\subseteq\t{Hom}\left(T,X\right)\right)\right\} 
\]
is a cover iff the collection of scheme morphisms $\left\{ U_{F_{i}}\rightarrow U_{F}\right\} $
is an fppf cover in $\aff$.
\begin{prop}
\label{prop:domran}The adjunction 
\[
\xymatrix{\pshv{\dom^{\Gamma}}\ar@<1ex>[r]^{\lke p} & \pshv{\dom}\ar@<1ex>[l]^{p_{*}}}
\]
induces mutually inverse equivalences after sheafification in the
fppf Grothendieck topology.
\end{prop}
The proof is given in \ref{sub:proof of prop domran}. The upshot
of this subsection is the following corollary:
\begin{cor}
\label{cor:The-upshot} There exists a naturally commuting triangle
\[
\xymatrix{\pshv{\dom^{\Gamma}}\ar[dr]_{\ \lke r} &  & \pshv{\dom}\ar[dl]^{\ \lke q}\ar[ll]_{p_{*}}\\
 & \shv{\aff;\t{fppf}}{}
}
\]
Consequently, for every $\mathcal{F}\in\pshv{\dom}$ D-module pullback
gives rise to an equivalence 
\[
\dmod\left(p_{*}\mathcal{F}\right)\xleftarrow{\cong}\dmod\left(\mathcal{F}\right)
\]

\end{cor}
The point being, that when formulating moduli problems of generic
data as functors of points, it suffices to describe their points over
$\dom^{\Gamma}$, rather than over the much larger category $\dom$.

\subsection{The Ran formulation of moduli problems of generic data}

\subsubsection{The monad}

Recall the functor $\Ran\xrightarrow{i}\dom^{\Gamma}$ introduced
in \ref{eq:functors}. Let $\mathcal{M}$ denote the monad on $\pshv{\aff}_{/\ran}\cong\pshv{\Ran}$
induced by the adjunction $\xymatrix{\pshv{\Ran}\ar@<1ex>[r]^{\lke i} & \pshv{\dom^{\Gamma}}\ar@<1ex>[l]^{i_{*}}}
$. I.e., its underlying endofunctor is $i_{*}\circ\ \lke i$, and its
unit and action transformations are induced by the adjunction unit
and co-unit.
\begin{rem}
\label{rem:action explicit}The action of $\mathcal{M}$ is pretty
simple. Let $\mathcal{G}\in\pshv{\Ran}$, let $\left(S,F\right)\in\Ran$,
and let us compute the value of $\mathcal{M}\left(\mathcal{G}\right)$
at $\left(S,F\right)$ (we implicitly identify the objects of $\Ran$
and $\dom^{\Gamma}$): 
\[
\mathcal{M}\left(\mathcal{G}\right)\left(S,F\right)=\lke i\mathcal{G}\left(S,F\right)=\colim\left(\left(\left(\Ran\right){}_{\left(S,F\right)/}\right)^{\t{op}}\xrightarrow{\mathcal{G}}\gpd_{\infty}\right)\cong
\]
Since both $\Ran$ and $\dom^{\Gamma}$ are fibered over $\aff$ we
can fix $S$ so that 

\[
\cong\colim\left(\left(\left(\Ran\left(S\right)\right){}_{\left(S,F\right)/}\right)^{\t{op}}\xrightarrow{\mathcal{G}}\gpd_{\infty}\right)\cong
\]
The category $\left(\left(\Ran\left(S\right)\right){}_{\left(S,F\right)/}\right)^{\t{op}}$
is discrete so 
\[
\cong\coprod_{G\subseteq F}\mathcal{G}\left(S,G\right)
\]

\medskip \noindent The following definition is the counterpart of
(the first part of) definition \ref{dom moduli problem}:\end{rem}
\begin{defn}
\label{def:ran formulation}The \emph{Ran formulation} for\emph{ moduli
problems of generic data} is the category of modules for the monad
$\mathcal{M}$ which we denote by 
\[
\t{Mod}_{\mathcal{M}}=\t{Mod}_{\mathcal{M}}\left(\pshv{\aff}_{/\ran}\right)
\]
\end{defn}
\begin{rem}
Recall the functors in diagram \ref{eq:functors}. Starting from a
presheaf $\mathcal{F}\in\pshv{\dom}$, we may construct two different
functors of points: 
\[
\lke q\mathcal{F}\,\,\,\t{and}\,\,\,\lke s\left(i_{*}p_{*}\mathcal{F}\right)
\]
The one on the left is the one we have been referring to as the associated
functor of points (cf. \ref{dom moduli problem}). It should be interpreted
as the functor of points obtained by quotienting out domain data.
The one on the right should be interpreted as the functor of points
obtained by retaining the domain data, but ``forgetting'' how to
restrict data to a smaller domain. 
\end{rem}
The simplicial resolution in the following theorem is the main result
of the section. We recall that the functors denoted $p,q,r,s$ and
$i$ were introduced in diagram \ref{eq:functors}. Also recall, that
the functor $\pshv{\Ran}\xleftarrow{i_{*}}\pshv{\dom^{\Gamma}}$ canonically
factors through $\t{Mod}_{\mathcal{M}}$.
\begin{thm}
\label{thm:ran_dom_equiv}The canonical functor 
\[
\xymatrix{\t{Mod}_{\mathcal{M}} & \pshv{\dom^{\Gamma}}\ar[l]}
\]
is an equivalence. 

For every $\mathcal{F}\in\pshv{\dom}$ there exists an augmented simplicial
object in $\pshv{\aff}$

\begin{equation}
\xymatrix{\cdots\ar@<2ex>[r]\ar@<-2ex>[r]\ar@{}@<0.5ex>[r]|(0.3){\fixedvdots} & \lke s\left(\mathcal{M}^{2}\left(i_{*}p_{*}\mathcal{F}\right)\right)\ar[r]\ar@<2ex>[r]\ar@<-2ex>[r] & \lke s\left(\mathcal{M}\left(i_{*}p_{*}\mathcal{F}\right)\right)\ar@<1ex>[r]\ar@<-1ex>[r]\ar@<1ex>[l]\ar@<-1ex>[l] & \lke s\left(\left(i_{*}p_{*}\mathcal{F}\right)\right)\ar[l]\ar[d]\\
 &  &  & \lke q\mathcal{F}
}
\label{eq:resolution}
\end{equation}
which becomes a colimit diagram, after sheafification in the fppf
Grothendieck topology.
\end{thm}
\noindent In the simplicial complex above, $\mathcal{M}$ refers
to the endofunctor $i_{*}\circ\ \lke i$ underlying the eponymous
monad acting on $\pshv{\aff}_{/\ran}$. We remark that despite the
vagueness in the existence statement of the simplicial complex, it
is actually quite explicit, as will be explained in \ref{sub:resolution}. 
\begin{proof}
The first assertion is a consequence of the Bar-Beck-Lurie theorem
\cite[Thm 6.2.0.6]{HA}, since $i_{*}$ is conservative and colimit
preserving (it admits a right adjoint given by right Kan extension).

For the second assertion, the Bar construction for $p_{*}\mathcal{F}\in\pshv{\dom^{\Gamma}}\cong\t{Mod}_{\mathcal{M}}$
yields an augmented simplicial complex in $\pshv{\dom^{\Gamma}}$
\[
\xymatrix{\cdots\ar@<2ex>[r]\ar@<-2ex>[r]\ar@{}@<0.5ex>[r]|(0.3){\fixedvdots} & \ \lke i\mathcal{M}^{2}\left(i_{*}p_{*}\mathcal{F}\right)\ar[r]\ar@<2ex>[r]\ar@<-2ex>[r] & \ \lke i\mathcal{M}\left(i_{*}p_{*}\mathcal{F}\right)\ar@<1ex>[r]\ar@<-1ex>[r]\ar@<1ex>[l]\ar@<-1ex>[l] & \ \lke i\left(i_{*}p_{*}\mathcal{F}\right)\ar[l]\ar[r] & p_{*}\mathcal{F}}
\]
which is a colimit diagram \cite[Thm 4.3.5.8 or Prop 6.2.2.12]{HA}. 

The sought after complex in $\pshv{\aff}$ is obtained by applying
the functor $\lke r$ (note that $r\circ i=s$), and composing the
augmentation with $\ \lke rp_{*}\mathcal{F}\rightarrow\ \lke q\mathcal{F}$,
to obtain 
\[
\xymatrix{\cdots\lke s\mathcal{M}^{2}\left(i_{*}p_{*}\mathcal{F}\right)\ar[r]\ar@<2ex>[r]\ar@<-2ex>[r] & \lke s\mathcal{M}\left(i_{*}p_{*}\mathcal{F}\right)\ar@<1ex>[r]\ar@<-1ex>[r]\ar@<1ex>[l]\ar@<-1ex>[l] & \lke s\left(i_{*}p_{*}\mathcal{F}\right)\ar[l]\ar[d]\\
 &  & \lke q\mathcal{F}
}
\]
This augmented complex becomes a colimit diagram after sheafification
in the fppf topology, since $\lke r$ is colimit preserving and by
Proposition \ref{prop:domran}. 
\end{proof}

\subsubsection{\label{sub:resolution}}

We make the resolution constructed in the theorem explicit. For the
sake of concreteness, let us consider the case $\mathcal{F}=\gmap Y_{\dom}\in\pshv{\dom}$.
Recall the presheaf $\gmap Y_{\ran}$ introduced in \ref{sub:monad Motivation},
and observe that 
\[
i_{*}p_{*}\gmap Y_{\dom}=\gmap Y_{\ran}\in\pshv{\aff}_{/\ran}
\]
We denote an $S$-point of $\gmap Y_{\ran}\times\left(\ran\right)^{n}$
by $\left(f;F_{0},\cdots,F_{n}\right)$, where it is understood that
each $F_{i}$ is a finite subset of $\t{Hom}\left(S,X\right)$, and
that $f$ is a generic map from $S\times X$ to $Y$, defined on the
open subscheme determined by $F_{0}$.

Using remark \ref{rem:action explicit}, we see that the $n$'th term
of the simplicial complex \ref{eq:resolution} is the subsheaf 
\[
\lke s\left(\mathcal{M}^{n}\left(\gmap Y_{\ran}\right)\right)\subseteq\gmap Y_{\ran}\times\left(\ran\right)^{n}
\]
whose $S$-points are the tuples $\left(f,;F\subseteq F_{1}\cdots\subseteq F_{n}\right)$
(i.e., in which the finite subsets are increasing) . The maps are
given as follows: 
\begin{enumerate}
\item For a degeneracy $\xymatrix{[n+1]\ar@{->>}[r]^{d_{i}} & [n]}
$ ($i,i+1\mapsto i$) we have 
\[
\xyR{0.5pc}\xymatrix{\lke s\left(\mathcal{M}^{n+1}\left(\gmap Y_{\ran}\right)\right) & \lke s\left(\mathcal{M}^{n}\left(\gmap Y_{\ran}\right)\right)\ar[l]\sp(0.5){}\\
\left(f;F_{0}\subseteq\ldots\subseteq F_{i}\subseteq F_{i}\subseteq\ldots\subseteq F_{n}\right) & \left(f;F_{0}\subseteq\ldots\subseteq F_{n}\right)\ar@{|->}[l]
}
\xyR{2pc}
\]

\item For a face map $\xymatrix{[n+1] & [n]\ar@{_{(}->}[l]\sb(0.4){s_{i}}}
$ (skip $i\in[n+1]$) we have 
\[
\xyR{0.5pc}\xymatrix{\left(\mathcal{M}^{n+1}\left(\gmap Y_{\ran}\right)\right)_{0}\ar[r]\sp(0.5){} & \left(\mathcal{M}^{n}\left(\gmap Y_{\ran}\right)\right)_{0}\\
(f;F_{0}\subseteq\ldots\subseteq F_{n+1})\ar@{|->}[r] & (f;F_{0}\subseteq\ldots\hat{F_{i}},\ldots\subseteq F_{n+1})
}
\xyR{2pc}
\]
where the hat over $\hat{F_{i}}$ denotes that the i'th term has been
omitted. We point out that, since $F_{i}\subseteq F_{i+1}$ the $i$'th
term in $(f;F_{0},\ldots,\hat{F_{i}},\ldots,F_{n+1})$ is the equal
to $F_{i}\cup F_{i+1}$, which how it should be morally interpreted.\end{enumerate}
\begin{rem}
There is another closely related way of describing the category $\t{Mod}_{\mathcal{M}}$.
The presheaf $\ran$ has the structure of a semi-group in presheaves
of sets, and the category $\t{Mod}_{\mathcal{M}}$ is equivalent to
a certain category of its modules (in $\gpd_{\infty}$). The approach
will be taken up a future note. 
\end{rem}

\subsection{D-module fully-faithfulness\label{sub:Dmod ff}}

In the proposition below we compare two categories of D-modules which
may be constructed from functors of points associated to a given moduli
problem. This result will be used in section \ref{sec:Cont_results}.
\begin{prop}
\label{prop:dmod ff}Let $\mathcal{F}\in\pshv{\dom^{\Gamma}}$. The
map on D-module categories induced by pullback along the adjunction
co-unit 
\[
\dmod\left(\lke s\left(\mathcal{F}\right)\right)\cong\dmod\left(\lke i\circ i_{*}\mathcal{F}\right)\xleftarrow{}\dmod\left(\mathcal{F}\right)
\]
is fully faithful.

Likewise, for $\mathcal{F}'\in\pshv{\dom}$ the functor \textup{
\[
\dmod\left(\lke s\left(p_{*}\mathcal{F}'\right)\right)\cong\dmod\left(\lke i\circ i_{*}\circ p_{*}\mathcal{F}'\right)\xleftarrow{}\dmod\left(\mathcal{F}'\right)
\]
is fully faithful.}
\end{prop}
In the proof we use the following general fact: if $C$ is an $\infty$-category,
then equivalences in $C$ satisfy ``2-out-of-6''. I.e., given a
commutative diagram in $C$ 
\[
\xymatrix{a\ar[rr]^{\cong}\ar[dr] &  & c\ar[dr]\\
 & b\ar[rr]_{\cong}\ar[ur] &  & d
}
\]
in which the horizontal morphisms are equivalences, we may conclude
that that all the morphisms are equivalences (the 6th being the composition
$a\rightarrow d$).
\begin{proof}
We start by reducing to the case when $\mathcal{F}$ is in the essential
image if the Yoneda functor $\dom^{\Gamma}\xrightarrow{}\pshv{\dom^{\Gamma}}.$
Denote the Yoneda image of point $\left(S,F\right)\in\dom^{\Gamma}$
by $\mathcal{Y}_{\left(S,F\right)}\in\pshv{\dom^{\Gamma}}$. Present
the presheaf $\mathcal{F}$ as the ``colimit of its points'' i.e.,
\[
\underset{\mathcal{Y}_{\left(S,F\right)}\rightarrow\mathcal{F}}{\colim}\left(\mathcal{Y}_{\left(S,F\right)}\right)\xrightarrow{\cong}\mathcal{F}
\]
Noting that both $\ \lke i$ and $i_{*}$ preserve colimits (since
both admit right adjoints), we also have 
\[
\underset{\mathcal{Y}_{\left(S,F\right)}\rightarrow\mathcal{F}}{\colim}\left(\ \lke i\circ i_{*}\mathcal{Y}_{\left(S,F\right)}\right)\xrightarrow{\cong}\ \lke i\circ i_{*}\mathcal{F}
\]
Consequently, it suffices to show that the functor 
\[
\lim_{\mathcal{Y}_{\left(S,F\right)}\rightarrow\mathcal{F}}\dmod\left(\ \lke i\circ i_{*}\mathcal{Y}_{\left(S,F\right)}\right)\xleftarrow{}\lim_{\mathcal{Y}_{\left(S,F\right)}\rightarrow\mathcal{F}}\dmod\left(\mathcal{Y}_{\left(S,F\right)}\right)
\]
is fully faithful. The latter will follow if we show that for every
$\left(S,F\right)\in\dom^{\Gamma}$ the functor 
\[
\dmod\left(\ \lke i\circ i_{*}\mathcal{Y}_{\left(S,F\right)}\right)\xleftarrow{}\dmod\left(\mathcal{Y}_{\left(S,F\right)}\right)
\]
is fully faithful, or equivalently that 
\[
\dmod\left(\lke s\left(\mathcal{Y}_{\left(S,F\right)}\right)\right)=\dmod\left(\lke{r\circ i}\circ i_{*}\mathcal{Y}_{\left(S,F\right)}\right)\xleftarrow{}\dmod\left(\lke{_{r}}\mathcal{Y}_{\left(S,F\right)}\right)=\dmod\left(S\right)
\]
is fully faithful. 

The latter functor is induced by the map in $\pshv{\aff}$ 
\[
\lke s\left(\mathcal{Y}_{\left(S,F\right)}\right)\rightarrow S
\]
The functor of points $\lke s\left(\mathcal{Y}_{\left(S,F\right)}\right)$
sends a scheme $T$, to the set 
\[
\left\{ \left(\left(T,G\right),T\xrightarrow{f}S\right)\,:\, G\subset\t{Hom}\left(T,X\right)\,\t{finite},\, G\supseteq f^{*}F\right\} 
\]
``Union with $F$'' gives rise to a map 
\[
\xyR{0.5pc}\xymatrix{\ran\times S\ar[r]\sp(0.5){\cup F} & \lke s\left(\mathcal{Y}_{\left(S,F\right)}\right)\\
\left(\left(T,G\right),T\xrightarrow{f}S\right)\ar@{|->}[r] & \left(T\xrightarrow{f}S,G\cup f^{*}F\right)
}
\xyR{2pc}
\]
which fits into the commutative diagram 
\[
\xymatrix{\left(\mathcal{Y}_{\left(S,F\right)}\right)_{0}\ar[rr]^{id}\ar[dr]_{\subseteq} &  & \left(\mathcal{Y}_{\left(S,F\right)}\right)_{0}\ar[dr]^{\rho}\\
 & \ran\times S\ar[rr]_{\pi_{2}}\ar[ur]^{\cup F} &  & S
}
\]
Passing to D-modules, pullback along the bottom map is fully-faithful
by \cite[Thm 1.6.5]{DG-Cont} (or \cite[Prop 4.3.3]{CA}). We conclude
by a ``2-out-of-6'' argument: for every pair $M,N\in\dmod\left(S\right)$,
the maps above give rise to a diagram of $\infty$-groupoids 
\[
\xymatrix{\t{Map}\left(\rho^{!}M,\rho^{!}N\right) &  & \t{Map}\left(\rho^{!}M,\rho^{!}N\right)\ar[dl]\ar[ll]_{=}\\
 & \t{Map}\left(\pi_{2}^{!}M,\pi_{2}^{!}N\right)\ar[ul] &  & \t{Map}\left(M,N\right)\ar[ll]_{\cong}\ar[ul]
}
\]
 By ``2-out-of-6'', for equivalences in $\gpd_{\infty}$, it follows
that $\t{Map}\left(\rho^{!}M,\rho^{!}N\right)\xleftarrow{}\t{Map}\left(M,N\right)$
is an equivalence of $\infty$-groupoids, so that $\dmod\left(i_{*}\mathcal{Y}_{\left(S,F\right)}\right)\xleftarrow{\rho^{!}}\dmod\left(S\right)$
is fully-faithful.

For $\mathcal{F}'\in\pshv{\dom}$, the second assertion now follows
from the fact that the functor 
\[
\dmod\left(p_{*}\left(\mathcal{F}'\right)\right)\xleftarrow{}\dmod\left(\mathcal{F}'\right)
\]
is an equivalence (proposition \ref{prop:domran}).
\end{proof}

\subsection{The proof of proposition \ref{prop:domran}\label{sub:proof of prop domran}}

The following lemma contains the geometric input for the proof of
proposition \ref{prop:domran}:
\begin{lem}
\label{lem:density}The functor $\dom^{\Gamma}\rightarrow\dom$ has
dense image with respect to the fppf topology. I.e., every point of
$\dom$ has a cover by points in the essential image of $\dom^{\Gamma}$.\end{lem}
\begin{proof}
Let $\left(S,U\right)\in\dom$; we must show that it admits a cover
by points in the essential image of $\dom^{\Gamma}$. 

We may assume that $S$ is connected, and by lemma \ref{lem:trivializing cover}
we may also assume that $U\subseteq S\times X$ is a divisor complement.
Let $\mathcal{L}\rightarrow\mathcal{O}_{S\times X}$ be an effective
Cartier divisor whose complement is $U$. 

Since $S$ is connected, the data of the divisor $\mathcal{L}\rightarrow\mathcal{O}_{S\times X}$
is equivalent to a map $S\rightarrow\mathcal{H}ilb_{X}^{n}$ for some
$n$, where $\mathcal{H}ilb_{X}^{n}$%
\footnote{Since $X$ is a curve, $\mathcal{H}ilb_{X}^{n}\cong X^{\left(n\right)}$,
the $n$'th symmetric power.%
} is the degree $n$ component of the Hilbert scheme of $X$. The standard
map $X^{n}\rightarrow\mathcal{H}ilb_{X}^{n}$ is an fppf cover (it
is faithfully flat and finite of index $n!$). Form the pullback 
\[
\xymatrix{\tilde{S}\ar[r]\ar[d] & X^{n}\ar[d]\\
S\ar[r] & \mathcal{H}ilb_{X}^{n}
}
\]
The components of the top map gives rise to a subset $F\subseteq\t{Hom}\left(\tilde{S},X\right)$,
which in turn determines a point $\left(\tilde{S},F\right)\in\dom^{\Gamma}$.
Observing that $U_{F}=\tilde{S}\times_{S}U$ we get a map $\left(\tilde{S},U_{F}\right)\rightarrow\Big(S,U\Big)$
which is an fppf cover in $\dom$, and whose domain is in the essential
image of $\dom^{\Gamma}$.
\end{proof}
Factor $p$ as 
\[
\dom^{\Gamma}\xrightarrow{p'}\dom^{00}\xrightarrow{j}\dom
\]
 where $\dom^{00}$ is the essential image of $p$ - the full subcategory
of $\dom$ consisting of ``graph complements''. We endow $\dom^{00}$
with the Grothendieck topology pulled back from the fppf topology
on $\dom$. We will prove that $j$ and $p'$ both induce equivalences
on sheaf categories, whence proposition \ref{prop:domran} will follow.

\medskip

Regarding $p^{'}$, informally, the idea is that every fiber of $p^{'}$
is weakly contractible, and that every map in such a fiber is a cover.
Thus, it is reasonable to suspect that $p^{'}$ might be a site equivalence.
The necessary accounting is a little involved, and the relevant site-theoretic
properties of $p^{'}$, which allow the argument to go through, are
embodied in the hypothesis of lemma \textbf{\ref{lem:dirty work}}.
Before stating the lemma, we introduce some notation:
\begin{notation}
\label{nota:slices and fibers}For a category $D$ and an object $d\in D$,
we use $D_{/d}$ to denote the overcategory, and we use $D_{d/}$
to denote the undercategory. We shall denote an object of $D_{/d}$
by $\left(d',d'\xrightarrow{}d\right)$ where $d'$ is an object of
$D$, and $d'\xrightarrow{}d$ is a morphism in $D$ (similarly for
undercategories).

If $C$ is another category and $C\xrightarrow{F}D$ is a functor,
$C_{d}$ denotes the \emph{fiber }of $f$\emph{ }over $d$\emph{ i.e.,
}the fibered product $C\times_{D}\{d\}$ in $\catinf$. We denote
$C_{/d}:=C\times_{D}D_{/d}$, it is a relative overcategory. We denote
an object of this category by the data $\left(c,F\left(c\right)\rightarrow d\right)$
where it is implicitly understood that $c$ is an object in $C$,
and that $F\left(c\right)\rightarrow d$ is a morphism in $D$. Dually,
we denote by $C_{d/}=C\times_{D}D_{d/}$, it is a relative undercategory.
This notation is slightly abusive since obviously these categories
are dependent on the functor $F$, and not only on $C$ and $d$.\end{notation}
\begin{lem}
\label{lem:dirty work}Let $C$ and $D$ be small sites whose underlying
categories admit all finite non-empty limits, and whose Grothendieck
topologies are generated by finite covers. Let $C\xrightarrow{p}D$
be a functor such that:
\begin{enumerate}
\item The Grothendieck topology on $C$ is the pullback of the topology
on D.
\item The functor $p$ is essentially surjective.
\item For every $c\in C$, and for every morphism in $D$, $d\xrightarrow{\tilde{f}}p\left(c\right)$,
there exists a morphism in $C$, $c'\xrightarrow{\tilde{f}}c$, which
lifts%
\footnote{But we do not assume that a Cartesian lift exists.%
} $f$. 
\item $p$ preserves finite limits.
\item For every $d\in D$, the functor $\left(C_{d}\right)^{\t{op}}\rightarrow\left(C_{d/}\right)^{\t{op}}$
is cofinal. 
\item For every $d\in D$, the category $C_{d}$ is a co-filtered poset.
\end{enumerate}
Then the functor 
\[
\shv C{}\xleftarrow{p_{*}}\shv D{}
\]
is an equivalence, and left Kan extension along $p$ is its inverse
(no sheafification necessary).
\end{lem}
In \ref{proof: prop:domran} we will show that $\dom^{\Gamma}\xrightarrow{p^{'}}\dom^{00}$
satisfies the hypothesis of this lemma.
\begin{proof}
We will show that the left Kan extension 
\[
\shv C{}\xrightarrow{\lke p}\pshv D
\]
lands in sheaves, and prove that the resulting adjoint functors $\left(\lke p,p_{*}\right)$
\[
\xymatrix{\shv C{}\ar@<1ex>[r]^{\lke p} & \shv D{}\ar@<1ex>[l]^{p_{*}}}
\]
are mutually inverse equivalences. 

The following is the key observation: Let $\mathcal{G}\in\shv C{}$,
and let $d\in D$. Then $\mathcal{G}$ is constant on the fiber $C_{d}$.
First we point out that (4) implies that $C_{d}$ admits all finite
non-empty limits, which may be computed in $C$. Let $c'\xrightarrow{f}c$
be a morphism in $C_{d}$; it is a cover by (1). The value of $\mathcal{G}$
at $c$ may be computed using the \v{C}ech complex of $f$. However,
(6) implies that this Cech complex is the constant simplicial object
with value $c^{'}$, since $c'\times_{c}c'=c'$ because $C_{d}$ is
a poset. It follows that $\mathcal{G}\left(c'\right)\xleftarrow{}\mathcal{G}\left(c\right)$
is an equivalence. Since $C_{d}$ it is weakly contractible (being
a co-filtered poset), the observation follows.

Let $\mathcal{G}\in\shv C{}$, and let us show that the co-unit transformation
(a-priori in $\pshv C$) 
\[
p_{*}\circ\lke p\mathcal{G}\rightarrow\mathcal{G}
\]
 is an equivalence. Fix $c\in C$, and let us prove that the map$ $
\[
p_{*}\circ\lke p\mathcal{G}\left(c\right)\rightarrow\mathcal{G}\left(c\right)
\]
 is an equivalence of groupoids. We compute 
\[
p{}_{*}\lke p\mathcal{G}\left(c\right)=\lke p\mathcal{G}\left(p\left(c\right)\right)=\colim\left(\left(C_{p\left(c\right)/}\right)^{\t{op}}\xrightarrow{\mathcal{G}}\gpd_{\infty}\right)\cong
\]
Since $\left(C_{p\left(c\right)}\right)^{\t{op}}\rightarrow\left(C_{p\left(c\right)/}\right)^{\t{op}}$
is co-final by (5), \textbf{
\[
\cong\colim\left(\left(C_{p\left(c\right)}\right)^{\t{op}}\xrightarrow{\mathcal{G}}\gpd_{\infty}\right)\cong
\]
}Because $\mathcal{G}$ is constant on the fibers, and these fibers
are weakly contractible we conclude 
\[
\cong\mathcal{G}\left(c\right)
\]

Next we show that for every $\mathcal{G}\in\shv C{}$, the presheaf
$\lke p\mathcal{G}$ is in fact a sheaf. Let $d\in D$, and let $\left\{ d_{i}\rightarrow d\right\} _{i=1}^{k}$
be a cover in $D$. Let $c\in C$ be such that $p\left(c\right)=d$,
and let $\left\{ c_{i}\xrightarrow{\tilde{f}_{i}}c\right\} _{i=1}^{k}$
be a lift of the $f_{i}$'s; it is a cover of $c$ by (1). For every
$n$-tuple of indexes in $\left\{ 1,\ldots,k\right\} $, $\overline{i}$,
we let $c_{\overline{i}}$ and $d_{\overline{i}}$ denote the corresponding
$n$-fold fibered products over $c$ and $d$, and we note that $p\left(c_{\overline{i}}\right)\cong d_{\overline{i}}$
by (4). Consequently, forming the \v{C}ech covers associated with
the covers, we obtain a commutative square 
\[
\xymatrix{{\displaystyle \lim_{[n]\in\Delta^{\t{op}}}\left(\coprod_{|\overline{i}|=n}\mathcal{G}\left(c_{\overline{i}}\right)\right)}\ar[d]^{\cong} & \mathcal{G}\left(c\right)\ar[d]^{\cong}\ar[l]\sb(0.35){\cong}\\
{\displaystyle \lim_{[n]\in\Delta^{\t{op}}}\left(\coprod_{|\overline{i}|=n}\lke p\mathcal{G}\left(d_{\overline{i}}\right)\right)} & \lke p\mathcal{G}\left(d\right)\ar[l]
}
\]
in which the vertical maps are equivalences by computation above,
and the top map is an equivalence because $\mathcal{G}$ is a sheaf.
We conclude that the bottom map is an equivalence for every cover
of $d$, thus $\lke p\mathcal{G}$ is a sheaf. 

We complete the proof of the lemma by observing that we have exhibited
adjoint functors 
\[
\xymatrix{\shv C{}\ar@<1ex>[r]^{\lke p} & \shv D{}\ar@<1ex>[l]^{p_{*}}}
\]
for which the co-unit transformation is an equivalence. In addition,
Since $p$ is essentially surjective, $p_{*}$ is conservative, whence
we conclude that the unit transformation is also an natural equivalence.
The equivalence of sheaf categories follows.
\end{proof}

\subsubsection{Proof of proposition \ref{prop:domran}\label{proof: prop:domran}}

Below, all sites are endowed with their (respective) fppf Grothendieck
topologies, and we suppress the topology in the notation. E.g., $\shv{\dom}{}:=\shv{\dom;\t{fppf}}{}$
etc.. 

Recall the factorization 
\[
\dom^{\Gamma}\xrightarrow{p'}\dom^{00}\xrightarrow{j}\dom
\]
We endow $\dom^{00}$ with the Grothendieck topology pulled back from
the fppf topology on $\dom$. We treat $p^{'}$ and $j$ separately.
\begin{enumerate}
\item We prove that $\shv{\dom^{00}}{}\xleftarrow{j_{*}}\shv{\dom}{}$ is
an equivalence by showing that it satisfies the hypothesis of a general
criterion for the inclusion of a sub-site to induce an equivalence
on sheaf categories (often referred to as the ``comparison lemma'').
A statement and proof of this criterion is included in the appendix
(lemma \ref{Comparison lemma}).

The functor $j$ has dense image by lemma \ref{lem:density}. The
category $\dom$ admits all finite limits. In particular, fibered
products in $\dom$ are given by squares of the form 
\[
\xymatrix{\left(R\times_{T}S,U\times_{T}W\right)\ar[r]\ar[d] & \left(R,W\right)\ar[d]^{g}\\
\left(S,U\right)\ar[r]^{f} & \left(T,V\right)
}
\]
 Whence it is evident whenever $U\subseteq S\times X$ and $W\subseteq T\times X$
are graph complements (i.e., present points in $\dom^{00}$), then
so is 
\[
U\times_{T}W\subseteq R\times_{S}T\times X
\]
whence it follows that $\left(R\times_{T}S,U\times_{T}W\right)\in\dom^{00}$.
These are precisely the hypothesis of the comparison lemma (\ref{Comparison lemma}),
and we conclude that $j_{*}$ is an equivalence of sheaf categories. 

\item We prove that $\shv{\dom^{\Gamma}}{}\xleftarrow{p_{*}^{'}}\shv{\dom^{00}}{}$
is an equivalence by showing that the functor $p^{'}$ satisfies the
hypothesis of lemma \ref{lem:dirty work}. Aside from (5), which we
will show, the rest of the hypothesis are immediate.

Fix $ $$\left(S,U\right)\in\dom^{00}$. In order to prove that 
\[
\left(\dom^{\Gamma}\right)_{\left(S,U\right)}^{\t{op}}\rightarrow\left(\left(\dom^{\Gamma}\right)_{\left(S,U\right)/}\right)^{\t{op}}
\]
is cofinal, it suffices to show that for every point $Q\in\left(\left(\dom^{\Gamma}\right)_{\left(S,U\right)/}\right)^{\t{op}}$
we have that the category 
\[
\left(\left(\dom^{\Gamma}\right)_{\left(S,U\right)}^{\t{op}}\right)_{Q/}
\]
is weakly contractible. Or equivalently, that its opposite category
\begin{equation}
\left(\left(\dom^{\Gamma}\right)_{\left(S,U\right)}\right)_{/Q}\label{eq:slice cat}
\end{equation}
is weakly contractible. The object $Q$ is presented by the data of
\[
\xymatrix{ & \left(T,G\right)\in\dom^{\Gamma}\ar@{|->}@<-4.5ex>[d]^{p^{'}}\\
\left(S,U\right)\ar[r]\sp(0.35){f:S\rightarrow T} & \left(T,U_{G}\right)\in\dom^{00}
}
\]
and the category \ref{eq:slice cat} classifies all the ways of lifting
$f$ to a ``commutative'' square 
\[
\xymatrix{\left(S,F\right)\ar[r]\ar@{|->}[d]^{p^{'}} & \left(T,G\right)\in\dom^{\Gamma}\ar@{|->}@<-4.5ex>[d]^{p^{'}}\\
\left(S,U\right)\ar[r]\sp(0.35){f:S\rightarrow T} & \left(T,U_{G}\right)\in\dom^{00}
}
\]
It is equivalent to the category of all finite subsets $F\subseteq\t{Hom}\left(S,X\right)$
whose associated open subscheme is $U$, and which contain $\left\{ g\circ f:T\rightarrow X\,:\, g\in G\right\} $,
with morphisms being the opposite of inclusion. This category is non-empty,
because the assumption that $\left(S,U\right)\in\dom^{00}$ implies
that it is the image of some $\left(S,F'\right)$, and then $\left(S,F'\cup f^{*}G\right)$
completes the square. It also admits finite products (given by the
union of $F$'s), thus is weakly contractible by \cite[lemma 2.4.6]{V}.\qed\end{enumerate}

\specialsection{\label{sec:Cont_results}Some ``Homological Contractibility'' Results}

In this section we present a few results which relate the D-module
categories associated to different moduli spaces of the kind we have
been considering. Namely, we prove that certain maps between the spaces
induce, via pullback, fully-faithful functors on D-module categories.
These results are of interest to the geometric Langlands program,
because the D-module categories involved are the counterparts, in
the geometric setting, of function spaces that appear on the automorphic
side of the correspondence in the classical setting. 

Fully faithfulness of D-module pullback has implications for classical%
\footnote{In contrast with ``higher'' invariants, such as D-module categories.%
}, invariants such as homology groups, and we start by pointing these
out in subsection \ref{sub:homology of a functor}. 

We emphasize the difference between the results we will discuss below,
and those discussed in subsection \ref{sub:Dmod ff}. Previously we
compared the D-module categories associated with different functor-of-points
formulations of the same moduli problem. Below we will compare D-module
categories associated to different moduli problems.

\subsection{The homology of a functor of points\label{sub:homology of a functor}}

In this subsection we define the homology groups of an arbitrary functor
of points, and relate this classical invariant to the higher invariant
$\dmod$.

\subsubsection{Motivation}

To every scheme $S$, of finite type over $\mathbb{C}$, we may associate
its analytic topological space, $S^{\t{an}}$. By the \emph{homology}
of the scheme $S$, we mean the topological (singular) homology of
$S^{\t{an}}$ with coefficients in $\mathbb{C}$. 

Let $\mathcal{F}\in\pshv{\aff}$ be any functor of points over $\mathbb{C}$.
We define the homotopy type of $\mathcal{F}$ to be the homotopy colimit,
over all the points of $\mathcal{F}$ 
\[
\t{type}\left(\mathcal{F}\right):=\underset{S\rightarrow\mathcal{F}}{\t{hocolim}}\left(S^{\t{an}}\right)
\]
It is the homology groups of this homotopy type which we are after
(when over $\mathbb{C}$). The point of the circuitous definition
for the homology of $\mathcal{F}$ given below, is to have it presented
in terms of D-module categories. In proposition \ref{prop:functor homology}\textbf{
}we prove that (over $\mathbb{C}$) both notions of homology agree. 
\begin{notation}
For a functor of points, $\mathcal{F}\in\pshv{\aff},$ and a pair
of D-modules $M,N\in\dmod\left(\mathcal{F}\right)$ we denote the
mapping space (an $\infty$-groupoid) by 
\[
\t{Map}_{\mathcal{F}}\left(M,N\right):=\t{Map}_{\dmod\left(\mathcal{F}\right)}\left(M,N\right)
\]

\end{notation}

\subsubsection{{}}

Let $\mathcal{F}\in\pshv{\aff}$ be an arbitrary functor of points,
and let 
\[
\mathcal{F}\xrightarrow{t}\spec\left(k\right)=:\t{pt}
\]
denote the map to the terminal object. We denote by $\mathbf{Vect}$
the stable $\infty$-category of chain complexes of vector spaces
over $k$, mod quasi-isomorphism (whose homotopy category is equivalent
to the derived category of the the ordinary category of $k$-vector
spaces). We shall identify $\dmod\left(\spec\left(k\right)\right)=\dmod\left(\t{pt}\right)\cong\vect$. 

A left adjoint, $t_{!}$, to the pullback functor $\dmod\left(\mathcal{F}\right)\xleftarrow{t^{!}}\vect$,
may not be globally defined, but nonetheless makes sense as a partial
functor, defined on the full subcategory of those $\mathcal{G}\in\dmod\left(\mathcal{F}\right)$
for which the functor 
\begin{equation}
\xyR{0.5pc}\xymatrix{\vect\ar[r] & \gpd_{\infty}\\
V\ar@{|->}[r] & \t{Map}_{\mathcal{F}}\left(\mathcal{G},t^{!}V\right)
}
\xyR{2pc}\label{eq:partial t!}
\end{equation}
is co-representable. For such $\mathcal{G}$, the object $t_{!}\mathcal{G}$
is such a co-representing object in $\vect$. 
\begin{defn}
The \emph{canonical sheaf} of a functor of points, $\mathcal{F}\in\pshv{\aff}$,
is 
\[
\omega_{\mathcal{F}}:=t^{!}k
\]
\end{defn}
\begin{lem}
\label{lem:homology well defined}Let $\mathcal{F}\in\pshv{\aff}$.
The partial functor $t_{!}$ is defined on $\omega_{\mathcal{F}}$.\end{lem}
\begin{proof}
Define an object of $\mathbf{Vect}$ 
\[
H:=\underset{S\xrightarrow{s}\mathcal{F}}{\colim}\left(t_{!}\omega_{S}\right)
\]
where the index diagram is the category of points of $\mathcal{F}$
(so each $S$ is an affine scheme). We remark that $t_{!}\omega_{S}\in\mathbf{Vect}$
is well-defined because $\omega_{S}$ is bounded holonomic.

We show that $H$ co-represents the functor \ref{eq:partial t!}.
Indeed 
\[
\begin{array}{c}
\t{Map}_{\mathcal{F}}\left(\omega_{\mathcal{F}},t^{!}V\right)\cong{\displaystyle \lim_{S\xrightarrow{s}\mathcal{F}}}\t{Map}_{S}\left(\omega_{S},t^{!}V\right)\cong{\displaystyle \lim_{S\xrightarrow{s}\mathcal{F}}}\t{Map}_{\t{pt}}\left(t_{!}\omega_{S},V\right)\cong\\
\hfill\t{Map}_{\t{pt}}\left(\underset{S\xrightarrow{s}\mathcal{F}}{\colim}\left(t_{!}\omega_{S}\right),V\right)=\t{Map}_{\t{pt}}\left(H,V\right)
\end{array}
\]
 we conclude that $t_{!}\omega_{\mathcal{F}}$ is defined. \end{proof}
\begin{defn}
We define the \emph{homology} of $\mathcal{F}$ to be
\[
\t H_{\bullet}\left(\mathcal{F};k\right):=t_{!}\omega_{\mathcal{F}}\in\mathbf{Vect}
\]

\end{defn}
\noindent It follows from the proof of lemma \ref{lem:homology well defined},
that 
\[
\t H_{\bullet}\left(\mathcal{F};k\right)\cong\underset{S\xrightarrow{s}\mathcal{F}}{\colim}\left(t_{!}\omega_{S}\right)=\underset{S\xrightarrow{s}\mathcal{F}}{\colim}\t H_{\bullet}\left(S;k\right)
\]

\medskip

\noindent The following well known proposition justifies our use
of the word 'homology' (we include a proof for completeness).
\begin{prop}
\label{prop:functor homology}Assume $k=\mathbb{C}$, and let $\mathcal{F}\in\pshv{\aff}$.
Then 
\[
\t{H_{\bullet}\left(\mathcal{F};\mathbb{C}\right)}\cong\t{H_{\bullet}^{top}\left(\t{type}\left(\mathcal{F}\right);\mathbb{C}\right)}
\]
where $\t H_{\bullet}^{\t{top}}$ denotes topological homology.\end{prop}
\begin{proof}
Since both homology theories are the left Kan extensions from affine
schemes (equivalently, they are colimit preserving), it suffices to
consider the case when $\mathcal{F}$ is representable by an affine
scheme $S$. 

For an affine scheme $S$, $\omega_{S}$ is a bounded holonomic complex,
and using the Riemann-Hilbert correspondence we obtain an equivalence
\[
\t H_{\bullet}\left(S;\mathbb{C}\right)=t_{!}t^{!}\mathbb{C}_{\t{pt}}\cong t_{!}^{c}t_{c}^{!}\mathbb{C}_{\t{pt}}\cong
\]
where $t_{c}^{!}$ and $t_{!}^{c}$ denote the $!$-functors on the
(derived) category of constructible sheaves of vector spaces on $S^{\t{an}}$.
Denote the duality functor on constructible sheaves by $\mathbb{D}$,
and topological co-homology by $\t H_{\t{top}}^{\bullet}$. By Verdier
duality we have an equivalence 
\[
\cong t_{!}^{c}t_{c}^{!}\mathbb{D}\mathbb{C}_{\t{pt}}\cong\mathbb{D}t_{*}^{c}t_{c}^{*}\mathbb{C}_{\t{pt}}=\mathbb{D}\left(\t H_{\t{top}}^{\bullet}\left(S^{\t{an}};\mathbb{C}\right)\right)\cong
\]
Using the universal coefficient theorem (and that $S^{\t{an}}$ has
finite dimensional co-homologies) we conclude 
\[
\cong\t H_{\bullet}^{\t{top}}\left(S^{\t{an}};\mathbb{C}\right)
\]
as claimed.\end{proof}
\begin{rem}
\label{rem:homology isom}If a map between functors of points $\mathcal{F}\xrightarrow{f}\mathcal{G}$
induces a fully faithful pullback functor on 
\[
\dmod\left(\mathcal{F}\right)\xleftarrow[\supseteq]{f^{!}}\dmod\left(\mathcal{G}\right)
\]
then 
\[
\t H_{\bullet}\left(\mathcal{F};k\right)\cong\t H_{\bullet}\left(\mathcal{G};k\right)
\]
since 
\[
\t{Map}_{\t{pt}}\left(t_{!}\omega_{\mathcal{F}},k\right)\cong\t{Map}_{\mathcal{F}}\left(\omega_{\mathcal{F}},\omega_{\mathcal{F}}\right)\xleftarrow[\cong]{f^{!}}\t{Map}_{\mathcal{G}}\left(\omega_{\mathcal{G}},\omega_{\mathcal{G}}\right)\cong\t{Map}_{\t{pt}}\left(t_{!}\omega_{\mathcal{G}},k\right)
\]
In the particular case of $\mathcal{F}\xrightarrow{t}\spec\left(k\right)$,
the fully faithfulness of $t^{!}$ is equivalent to $\t H_{\bullet}\left(\mathcal{F};k\right)\cong t_{!}t!\left(k_{\t{pt}}\right)\rightarrow k_{\t{pt}}$
being an equivalence.
\end{rem}
{}

\bigskip

In a certain sense the following remark summarizes the essential thesis
of this paper:
\begin{rem}
For each concrete moduli problem of generic data we have introduced
numerous candidates for a functor of points presenting a moduli space
- with or without domain data, allowing general domains or only graph-complements,
presheaves over $\ran$ and modules for $\mathcal{M}$ therein, and
the compactification constructions of sections 3 and 4. The underlying
theme of this paper is, that in many ways, these different choices
are all equivalent. E.g., a main application of propositions \ref{cor:The-upshot},
\ref{prop:dmod ff}, \ref{prop:quasimap equiv}, and \ref{prop:B_bar points equiv},
about fully-faithfulness of D-module pullback, is that the homology
groups of all these functors of points are isomorphic.
\end{rem}

\subsection{Back to D-modules}

The following theorem of Gaitsgory is the prototype for the main result
of this section, as well as its foundation:
\begin{thm}
\cite[Thm 1.8.2]{DG-Cont} \label{thm:Gaits-contr}Let $Y$ be a connected
affine scheme which can be covered by open subschemes $U_{\alpha}$,
each of which is isomorphic to an open subscheme of the affine space
$\mathbb{A}^{n}$ (for some integer $n$). Then, the pullback functor
\[
\dmod\left(\gmap Y_{\ran}\right)\xleftarrow{t^{!}}\dmod\left(\spec\left(k\right)\right)=\mathbf{Vect}
\]
is fully faithful.\qed
\end{thm}
In particular, we conclude that under the assumptions of the theorem
\[
\t H_{\bullet}\left(\gmap Y_{\ran};k\right)\cong k
\]
 (see \ref{rem:homology isom}).

In this section we use theorem \ref{thm:Gaits-contr} to obtain more
results of a similar nature. 

\medskip

Recall that the a-priori premise of this paper is that ``the correct''
(from a conceptual point of view) space of generic maps is presented
by the functor of points $\gmap Y$, introduced in \ref{exam:genmaps}. 
\begin{cor}
Let $Y$ be as in theorem \ref{thm:Gaits-contr}. The pullback functor
\[
\dmod\left(\gmap Y\right)\xleftarrow{}\dmod\left(\spec\left(k\right)\right)=\mathbf{Vect}
\]
is fully faithful.\end{cor}
\begin{proof}
Consider the pull back functors 
\[
\xyR{1pc}\xyC{1pc}\xymatrix{ &  & \mathbf{Vect}\ar[d]^{\cong}\\
\dmod\left(\lke rp_{*}\gmap Y_{\dom}\right)\ar[d]\sb(0.5){\beta^{!}} & \dmod\left(\gmap Y\right)\ar[l]\sb(0.4){\alpha^{!}} & \dmod\left(\spec\left(k\right)\right)\ar[l]\sb(0.4){t^{!}}\\
\dmod\left(\gmap Y_{\ran}\right)
}
\xyR{2pc}\xyC{2pc}
\]
The composition is fully faithful by theorem \ref{thm:Gaits-contr}.
$\alpha^{!}$ is an equivalence by proposition \ref{prop:domran},
and $\beta^{!}$ is fully faithful by proposition \ref{prop:dmod ff}.
We conclude that $t^{!}$ is fully faithful.
\end{proof}
The next result is a minor extension of theorem \ref{thm:Gaits-contr},
in which we remove the requirement that the target be affine:
\begin{thm}
\label{thm:Y not affine}Let $Y$ be a connected and separated scheme
which can be covered by open subschemes $U_{\alpha}$, each of which
is isomorphic to an open subscheme of the affine space $\mathbb{A}^{n}$
(for some integer $n$). Then, the pullback functor 
\[
\dmod\left(\gmap Y\right)\xleftarrow{t^{!}}\dmod\left(\spec\left(k\right)\right)=\mathbf{Vect}
\]
is fully faithful. When $Y$ is projective, $t^{!}$ admits a (globally
defined) left adjoint.
\end{thm}
The main examples to consider for $Y$ (aside from $\mathbb{A}^{n}$),
are $\mathbb{P}^{n}$, a connected affine algebraic group $G$, and
its flag variety $G/B$. We prove this theorem in \ref{proof:thm:Y not affine}.

\medskip

Recall the functors of points $\bun[H\left(\t{gen}\right)]$ and $\bun[1\left(\t{gen}\right)]\in\pshv{\aff}$
which were introduced in \ref{exam:bungbk}. The following theorem
is the main result of this section:
\begin{thm}
\label{thm:more_cont}$\vphantom{}$ Let $G$ be a connected reductive
algebraic group. Let $H$ be a subgroup of $G$ such that $G/H$ is
rational (e.g., $H=1$, $N$, or any parabolic subgroup). Then, the
pull back functor 
\[
\dmod\left(\bun[H\left(\t{gen}\right)]\right)\xleftarrow{}\dmod\left(\bun\right)
\]
is fully faithful. When $H=B$, this pullback functor admits a (globally
defined) left adjoint.
\end{thm}
\noindent Theorem \ref{thm:more_cont} is proven in \ref{proof:thm more cont},
after some preparations.

The existence of the left adjoint can be extended to include any parabolic
subgroup, if the statement (and proof) of proposition \ref{prop:B_bar points equiv}
is extended accordingly.

We remark that the existence of the left adjoint above (and in theorem
\ref{proof:thm:Y not affine}) is a kind of ``properness'' of the
map $\bun[B\left(\t{gen}\right)]\rightarrow\bun$ ($\gmap Y\xrightarrow{}\spec\left(k\right)$),
though this map between functor of points is not schematic. We also
emphasize, as a concrete application, that pullback fully faithfulness
implies that the homology of the spaces in question is equivalent
(see remark \ref{rem:homology isom}).

\medskip

The rest of this section is contains the proofs (and supporting lemmas)
of theorems \ref{thm:Y not affine} and \ref{thm:more_cont}.

By a \emph{Zariski cover} of presheaves we mean a morphism of presheaves,
which becomes an effective epimorphism after sheafification in the
Zariski Grothendieck topology. 
\begin{lem}
\label{lem:zariski covers}The functor 
\[
\gmap -:\S\rightarrow\pshv{\aff}
\]
 carries Zariski covers to Zariski covers.\end{lem}
\begin{proof}
Let $Y$ be a scheme, and $\left\{ Y_{i}\rightarrow Y\right\} _{i\in I}$
its finite cover by open subschemes. We must show that for every point
$S\xrightarrow{s}\gmap Y$, there exists a Zariski cover $\tilde{S}\rightarrow S$,
and a lift 
\[
\xymatrix{ &  & {\displaystyle \coprod_{i\in I}}\gmap{Y_{i}}\ar[d]\\
\tilde{S}\ar[r]\ar@{-->}[urr] & S\ar[r] & \gmap Y
}
\]

The point $s$ is presented by a point point $\left(S,U\right)\in\dom$,
together with a regular map $U\rightarrow Y$. For every $i\in I$,
let $U_{i}:=U\times_{Y_{i}}Y\subseteq U$ (it is an open subscheme
of $U$), and let $S_{i}\subseteq S$ be the open subscheme which
is the image of $U_{i}\rightarrow S\times X\rightarrow S$. The composition
$U_{i}\rightarrow U\rightarrow Y$ lands in $Y_{i}$, and thus determines
a lift 
\[
\xymatrix{ &  & \gmap{Y_{i}}\ar[d]\\
S_{i}\ar[r]\ar[urr] & S\ar[r] & \gmap Y
}
\]
Taking $\tilde{S}=\coprod S_{i}$, the map $\coprod S_{i}\rightarrow\coprod\gmap{Y_{i}}$
is the sought after lift of $s$. \end{proof}
\begin{lem}
\label{lem:limit preservation}The functor 
\[
\gmap -:\aff\rightarrow\pshv{\aff}
\]
 preserves finite limits.\end{lem}
\begin{proof}
$\gmap -$ is the composition 
\[
\aff\xrightarrow{\gmap -_{\dom}}\pshv{\dom}\xrightarrow{\ \lke q}\pshv{\aff}
\]
 $\gmap -_{\dom}$ preserves (all) limits, and $\ \lke q$ preserves
finite limits.
\end{proof}

\subsubsection{Proof of Theorem \ref{thm:Y not affine}\label{proof:thm:Y not affine}}

The theorem is now an almost immediate result of lemmas \ref{lem:limit preservation}
and \ref{lem:zariski covers}. 

Let $\left\{ U_{i}\rightarrow Y\right\} _{i\in I}$ be a cover of
$Y$ by its affine open subschemes, which are each isomorphic to an
open subscheme of $\mathbb{A}^{n}$. We note that since $Y$ is separated,
every intersection of the $U_{i}$'s has the same property. 

Construct the \v{C}ech complex corresponding to the cover 
\[
\xyR{0.5pc}\xyC{4pc}\xymatrix{\Delta^{\t{op}}\ar[r]\sp(0.5){U_{\bullet}} & \aff\\
{}[n]\ar@{|->}[r] & {\displaystyle \coprod_{|\overline{i}|=n}U_{\overline{i}}}
}
\xyR{2pc}\xyC{2pc}
\]
where $\overline{i}=\left(i_{1},\ldots,i_{n}\right)$ is a multi-index
of elements in $I$, and $U_{\overline{i}}=\cap_{k=1}^{n}U_{i_{k}}$.
We have that $Y=\underset{[n]\in\Delta^{\t{op}}}{\colim}\left({\displaystyle \coprod_{|\overline{i}|=n}U_{\overline{i}}}\right)$.
By lemma \ref{lem:limit preservation}, the simplicial object 
\[
\xyR{0.5pc}\xyC{4pc}\xymatrix{\Delta^{\t{op}}\ar[r]\sp(0.4){\gmap{U_{\bullet}}} & \pshv{\aff}\\
{}[n]\ar@{|->}[r] & {\displaystyle \coprod_{|\overline{i}|=n}\gmap{{\displaystyle U_{\overline{i}}}}}
}
\xyR{2pc}\xyC{2pc}
\]
is the \v{C}ech nerve of $\left\{ \gmap{U_{i}}\xrightarrow{}\gmap Y\right\} _{i\in I}$,
which is a Zariski cover by lemma \ref{lem:zariski covers}. We conclude
that the homology of $\gmap Y$ is isomorphic to that of a point,
being the colimit

\begin{multline*}
\t H_{\bullet}\left(\gmap Y;k\right)\cong\underset{[n]\in\Delta^{\t{op}}}{\colim}\t H_{\bullet}\left(\coprod_{|\overline{i}|=n}\gmap{U_{\overline{i}}};k\right)\cong\\
\cong\underset{[n]\in\Delta^{\t{op}}}{\colim}\t H_{\bullet}\left(\coprod_{|\overline{i}|=n}\spec\left(k\right);k\right)\cong\t H_{\bullet}\left(\spec\left(k\right);k\right)
\end{multline*}
Finally, the equivalence $\t H_{\bullet}\left(\mathcal{F}\right)\cong\t H_{\bullet}\left(\spec\left(k\right);k\right)$
implies that the fully faithfulness of $t^{!}$ (see remark \ref{rem:homology isom}).

Regarding the existence of the left adjoint, when $Y$ is projective,
this is a restatement of corollary\textbf{ \ref{cor:dmod equiv gmap}
}(2). \qed

\bigskip

We continue with the preparations for the proof of the theorem \ref{thm:more_cont}.
The following is a corollary of lemma \ref{lem:limit preservation}:
\begin{cor}
Let be $G$ an algebraic group. 
\begin{enumerate}
\item $\gmap G$ is a group object in $\pshv{\aff}$.
\item If $\, Y$ is a scheme acted on by $G$, then $\gmap Y$ is acted
on by $\gmap G$.
\end{enumerate}
\end{cor}
\begin{defn}
A map of presheaves $\mathcal{E}\rightarrow\mathcal{B}$ is an \emph{fppf-locally
trivial fibration with fiber $\mathcal{F}$,} if there exists an fppf
cover $\mathcal{B}^{'}\rightarrow\mathcal{B}$ (i.e., a morphism of
presheaves which becomes an effective epimorphism after fppf sheafification),
and a map 
\[
\mathcal{B}^{'}\times_{\mathcal{B}}\mathcal{E}\rightarrow\mathcal{F}
\]
which exhibits the former as a product $\mathcal{B}^{'}\times_{\mathcal{B}}\mathcal{E}\cong\mathcal{F}\times\mathcal{B}^{'}$.\end{defn}
\begin{lem}
\label{lem:fibdmod}Let $\mathcal{E}\xrightarrow{p}\mathcal{B}$ be
an fppf-locally trivial fibration with fiber $\mathcal{F}$ of presheaves
in $\pshv{\aff}$. 

If $\dmod\left(\mathcal{F}\right)\xleftarrow{t^{!}}\dmod\left(\spec\left(k\right)\right)$
is fully faithful, then $\dmod\left(\mathcal{E}\right)\xleftarrow{p^{!}}\dmod\left(\mathcal{B}\right)$
is fully faithful.\end{lem}
\begin{proof}
Let $M,N\in\dmod\left(\mathcal{B}\right)$. We must show that 
\[
(*)\,\,\,\,\,\,\,\,\t{Map}_{\mathcal{E}}\left(p^{!}M,p^{!}N\right)\xleftarrow{}\t{Map}_{\mathcal{B}}\left(M,N\right)
\]
 is an equivalence of $\infty$-groupoids. 

Fix a Cartesian square 
\[
\xymatrix{\mathcal{F}\times\mathcal{B}_{0}\ar[r]\ar[d]\ar@<-1ex>@{}[dr]|(0.2){\ulcorner} & \mathcal{E}\ar[d]\\
\mathcal{B}_{0}\ar[r] & \mathcal{B}
}
\]
in which $\mathcal{B}_{0}\rightarrow\mathcal{B}$ is an fppf cover.
Denote the \v{C}ech simplicial complex associated with the cover
$\mathcal{B}_{0}\rightarrow\mathcal{B}$ by 
\[
\xyR{0.5pc}\xymatrix{\Delta^{\t{op}}\ar[r]\sp(0.4){\mathcal{B}_{\bullet}} & \pshv{\aff}\\
{}[n]\ar@{|->}[r] & \mathcal{B}_{n}\save+<1.8cm,-0.25cm>*{:=\underbrace{\mathcal{B}_{0}\times_{\mathcal{B}}\cdots\times_{\mathcal{B}}\mathcal{B}_{0}}_{n-\t{times}}}\restore
}
\xyR{2pc}
\]
and the one associated with the cover $\mathcal{B}_{0}\times\mathcal{F}\rightarrow\mathcal{E}$
by 
\[
\xyR{0.5pc}\xymatrix{\Delta^{\t{op}}\ar[r]\sp(0.35){\left(\mathcal{B}_{0}\times\mathcal{F}\right)_{\bullet}} & \pshv{\aff}\\
{}[n]\ar@{|->}[r] & \left(\mathcal{B}_{0}\times\mathcal{F}\right)_{n}\save+<3.2cm,-0.25cm>*{:=\underbrace{\left(\mathcal{B}_{0}\times\mathcal{F}\right)\times_{\mathcal{E}}\cdots\times_{\mathcal{E}}\left(\mathcal{B}_{0}\times\mathcal{F}\right)}_{n-\t{times}}}\restore
}
\xyR{2pc}
\]
There exist an equivalences of stable $\infty$-categories 
\[
\dmod\left(\mathcal{B}\right)\cong\lim_{[n]\in\Delta}\dmod\left(\mathcal{B}_{n}\right)\,\,\,\t{and}\,\,\,\dmod\left(\mathcal{E}\right)\cong\lim_{[n]\in\Delta}\dmod\left(\left(\mathcal{B}_{0}\times\mathcal{F}\right)_{n}\right)
\]
and $p^{!}$ is induced by a transformation of the co-simplicial diagrams.

Let $M_{n}$ and $N_{n}$ denote the images of $M$ and $N$ in $\dmod\left(\mathcal{B}_{n}\right)$.
Let $\left(p^{!}M\right)_{n}$ and $\left(p^{!}N\right)_{n}$ denote
the images of $M$ and $N$ in $\dmod\left(\left(\mathcal{B}_{0}\times\mathcal{F}\right)_{n}\right)$.
We have equivalences of $\infty$-groupoids 
\[
\t{Map}_{\mathcal{B}}\left(M,N\right)\cong\lim_{[n]\in\Delta}\t{Map}_{\mathcal{B}_{n}}\left(M_{n},N_{n}\right)
\]
and 
\[
\t{Map}_{\mathcal{E}}\left(p^{!}M,p^{!}N\right)\cong\lim_{[n]\in\Delta}\t{Map}_{\left(\mathcal{B}_{0}\times_{\mathcal{B}}\mathcal{E}\right)_{n}}\left(\left(p^{!}M\right)_{n},\left(p^{!}N\right)_{n}\right)
\]
We have that $\left(p^{!}M\right)_{n}\cong p_{n}^{!}N_{n}$, where
$p_{n}$ is the map $ $$\left(\mathcal{B}_{0}\times_{\mathcal{B}}\mathcal{E}\right)_{n}\xrightarrow{p_{n}}\mathcal{B}_{n}$.
Furthermore, the map $(*)$ is the limit of the maps 
\[
(**)\,\,\,\,\,\,\,\t{Map}_{\left(\mathcal{B}_{0}\times_{\mathcal{B}}\mathcal{E}\right)_{n}}\left(p_{n}^{!}M_{n},p_{n}^{!}N_{n}\right)\xleftarrow{}\t{Map}_{\mathcal{B}_{n}}\left(M_{n},N_{n}'\right)
\]
Finally, observing that for each $n$ we have a commuting diagram
\[
\xymatrix{\left(\mathcal{B}_{0}\times_{\mathcal{B}}\mathcal{E}\right)_{n}\ar[r]^{\cong}\ar[d]^{p_{n}} & \mathcal{F}\times\mathcal{B}_{n}\ar[d]\\
\mathcal{B}_{n}\ar[r]^{=} & \mathcal{B}_{n}
}
\]
which implies that, for every $[n]\in\Delta$, the functor $p_{n}^{!}$
is fully faithful.Thus, the map $(**)$ is an equivalence, whence
we conclude that the map $(*)$ is an equivalence.
\end{proof}
The proof below uses the symmetric properties of the map $\bun[H\left(\t{gen}\right)]\rightarrow\bun$.
After the proof we indicate a strategy for another proof, similar
to that of theorem \ref{thm:Y not affine}..

\subsubsection{Proof of Theorem \ref{thm:more_cont}\label{proof:thm more cont}{}}

Observe that there exists  a Cartesian square
\[
\xymatrix{\gmap{G/H}\times\bun[1\left(\t{gen}\right)]\ar[d]_{p}\ar[r]\sp(0.65){a} & \bun[H\left(\t{gen}\right)]\ar[d]\\
\bun[1\left(\t{gen}\right)]\ar[r] & \bun
}
\]

The functor $\bun[1\left(\t{gen}\right)]\rightarrow\bun$ becomes
an effective epimorphism after Etale sheafification. Indeed, if $\mathcal{P}_{G}$
is $G$-torsor on $S\times X$ then, by the Drinfeld-Simpson theorem
\cite[Thm 2]{DS}, there exists an Etale base change $S'\rightarrow S$,
such that $\mathcal{P}_{G}\times_{S}S^{'}$ is Zariski locally trivial,
hence admits a generic trivialization.

Our assumptions on $H$ imply that it may be covered by open subschemes
which are isomorphic to open subschemes of affine space. Thus by theorem
\ref{thm:Gaits-contr}, 
\[
\dmod\left(\gmap H\right)\xleftarrow{}\dmod\left(\spec\left(k\right)\right)
\]
is fully faithful. The fully faithfulness of the pullback functor
\[
\dmod\left(\bun[H\left(\t{gen}\right)]\right)\xleftarrow{}\dmod\left(\bun\right)
\]
now follows from \ref{lem:fibdmod}.

In the case when $H=B$, the existence of a left adjoint is a restatement
of corollary \ref{cor:dmod equiv} (2). \qed
\begin{rem}
A different proof of the theorem may be deduced from the following
statement:

Let $Y\rightarrow S\times X$ be an fppf fiber bundle with fiber $F$,
which becomes Zariski-locally trivial, after a suitable fppf base
change $\tilde{S}\rightarrow S$. And such that $\dmod\left(F\right)\xleftarrow{}\dmod\left(\spec\left(k\right)\right)$
is fully faithful.

Then the pullback functor 
\[
\dmod\left(\gsect[][S]{S\times X}Y\right)\xleftarrow{}\dmod\left(S\right)
\]
is fully faithful.\end{rem}

\section{\label{sec:Abstract-nonsense}Appendix A - The quasi functor $\dmod$
and other abstract nonsense}

The main purpose of this section is to sketch out the construction
of the functorial assignment
\[
\mathcal{F}\mapsto\dmod\left(\mathcal{F}\right)
\]
There are two catches, the first is that we wish to allow $\mathcal{F}\in\pshv{\aff}$
to be an arbitrary functor of points without assuming any representability
properties (such as being a scheme or Artin stack). The second is
that by $\dmod\left(\mathcal{F}\right)$ we mean a \emph{stable }$\infty$-category
whose homotopy category is the eponymous triangulated category which
is usually considered. 

What follows is intended as a broad overview only; we shall point
to references where  details may be found.

\subsection{Motivation}

We think of the quasi-functor 
\[
\Sop\xrightarrow{\dmod^{\triangle}}\mbox{\ensuremath{\left\{  \t{triangulated\,\ cateogries}\right\} } }
\]
which assigns to a scheme $S$ its category of D-module sheaves, as
an invariant defined on schemes whose values are triangulated categories.
$\qco^{\triangle}$, which assigns the category of quasi-coherent
sheaves, and $\indco^{\triangle}$ which assigns the category of ind-coherent
sheaves are other triangulated category valued invariants commonly
considered in the study of schemes. We wish to study spaces presented
by arbitrary functors of points (such as the ones considered in this
paper) by the same kind of invariants, thus we wish to extend their
domains of definition to include all functors of points - representable
or not. 

Consider first a diagram of schemes $I\xrightarrow{F}\S$, where $I$
is a small index category. Informally, the value of $\dmod^{\triangle}$
on such a diagram is $\lim_{I}\dmod^{\triangle}\left(F\left(i\right)\right)$.
As a first approximation, this limit is the category whose objects
consist of an assignment of D-modules 
\[
i\mapsto\mathcal{F}_{i}\in\dmod^{\triangle}\left(F\left(i\right)\right),
\]
 which is required to be pullback compatible in the sense that for
every morphism $i\xrightarrow{f}j$ in $I$, the D-modules $\mathcal{F}_{i}$
and $F\left(f\right)^{*}\left(\mathcal{F}_{j}\right)$ on the scheme
$F\left(i\right)$ are identified (compatibly). 

However, the totality of triangulated categories is not suited for
taking such limits. In practice, we take a kind of homotopy limit
and it is the formalism of this process which is the topic of this
section. This homotopy limit will manifest itself as a (coherent)
limit within the $\infty$-category of \emph{stable} $\infty$-categories
(see below).

\subsection{$\infty$-categories}

We use the term \emph{$\infty$-category} to refer to the abstract
(model independent) notion of an $\left(\infty,1\right)$-category
- a collection of objects with an infinite hierarchy of morphisms,
in which all n $n$-morphisms for $n\geq2$ are invertible. For the
most part, we use the notation of \cite{HT} and \cite{HA}.

To an $\infty$-category, $C$, one can associate an ordinary category
$hC$ - its homotopy category. It is the left adjoint to the inclusion
of ordinary categories in $\infty$-categories (\cite[1.2.3]{HT}).

Being a \emph{stable $\infty$-category} is a property an $\infty$-category
(rather than structure, see \cite[1.1.1.14]{HA}). In setting of $\infty$-categories,
these stable categories play a role similar to that of abelian categories
in the ordinary setting. The homotopy category of a stable $\infty$-category
is naturally a triangulated category (\cite[1.1.2.13]{HA}). As such,
we think of a stable $\infty$-category as remembering higher morphisms
which its triangulated homotopy category has forgotten.\emph{ Exactness
}is a property of functors between stable $\infty$-categories (\cite[1.1.4]{HA}).
The functor it induces between the triangulated homotopy categories
is then a triangulated functor. 
\begin{rem}
Set theory must be dealt with in some way. We consider three sizes
of sets: small, large and abominable%
\footnote{cf. french alpine grades. %
}. One can make sense of these by considering a chain of universes
of sets $\mathcal{U}\subseteq\hat{\mathcal{U}}\subseteq\hat{\hat{\mathcal{U}}}$,
each a set in its super universes. Set size corresponds to universe
size (small sets are the objects of $\mathcal{U}$, which is itself
a large set. Large sets are the objects of $\hat{\mathcal{U}}$, which
is itself an abominable set etc). 

It is the theory of $\mathcal{U}$ we are interested in, the others
are auxiliary universes introduced to make sense of ``large constructions''.
In particular $k$ is a field in $\mathcal{U}$, and $\S$ refers
to the category of finite type schemes over $k$ which belong to $\mathcal{U}$.
Note that due to the ``finite type'' restriction, $\S$ is an essentially
small category.
\end{rem}
A large $\infty$-category may have the property of being \emph{presentable}.
For the precise definition see \cite[section 5.5]{HT}. Let us only
say that this property allows us to deal with a large category from
within the small universe and that most large categories we shall
deal with will shall have this property.

\subsubsection{Categories of $\infty$-cat's}

The totality of all small $\infty$-categories is itself given the
structure of an (large) $\infty$-category, denoted $\catinf$. We
denote by $\gpdinf\subseteq\catinf$ the full subcategory (small)
$\infty$-groupoids, also known as spaces or homotopy-types. 

Likewise, the totality of all large $\infty$-categories is denoted
$\cathat$. Let $\catexl\subseteq\cathat$ be the full subcategory
of presentable stable categories and functors which preserve small
colimits (equivalently, are left adjoints, whence the $L$ in the
notation).
\begin{prop}
\cite[1.1.4.4]{HA}\cite[5.5.3.13]{HT}\label{pro:limits in stab}The
(abominably large) infinity category $\catexl$ admits all small limits.
its inclusion into $\cathat$ is limit preserving. \qed
\end{prop}

\subsection{Stable invariants\label{sub:Stable-observables}}

Classically one considers triangulated invariants (i.e. $\t{Tricat}$
valued). However, the $\left(2,1\right)$-category $\t{Tricat}$ is
not well suited for taking limits as we wish to do. In order to remedy
this shortcoming we compute our limits in $\catexl$. The stable and
triangulated invariants are related as follows:
\begin{defn}
A \emph{stable invariant} of schemes is a functor $\affop\xrightarrow{\mathcal{A}}\catexl$.
Its associated Triangulated invariant is defined to be the composition
\[
\affop\xrightarrow{\mathcal{A}}\catexl\xrightarrow{h}\mbox{Tricat}
\]

\end{defn}
The particular example we consider in this paper is when $\mathcal{A}$
is $\dmod$, but along the way we will also mention $\indco$. In
these cases the associated triangulated invariant is a familiar notion,
and we would like to have a stable model for it. I.e. a lift

\[
\xymatrix{ & \catexl\ar[d]^{h}\\
\affop\ar[r]^{A}\ar@{-->}[ur] & \mbox{Tricat}
}
\]
It seems to be the case that for the invariants mentioned above, such
a lift exists and is in fact quite natural.

\subsubsection{IndCoh}

Our initial input for the construction of $\dmod$, is the functor
which assigns to a scheme its category of ind-coherent sheaves. Ind-coherent
sheaves are a convenient substitute for quasi-coherent sheaves, for
an in-depth discussion see \cite{DG-Indcoh}. 

In \cite[Cor 3.3.6]{DG-Indcoh} Gaitsgory constructs a functor 
\[
\indco:\aff\rightarrow\catexl
\]
assigning to a scheme a higher category model of its triangulated
category of ind-coherent sheaves, and to every morphism of schemes
the $!-$pullback. There, pre-triangulated DG-categories are taken
as models for stable $\infty$-categories. 

\medskip

We extend $\indco$ to arbitrary functors of points via a right Kan
extension 
\[
\xymatrix{\affop\ar[d]^{j}\ar[r]^{\indco} & \catex\\
\mathcal{P}shv\left(\aff\right)^{\t{op}}\ar@{-->}[ur]_{\indco:=\t{RKE}}
}
\]
where $j$ is the Yoneda embedding. We abusively continue to denote
the extension by $\indco$. So defined, $\indco$ preserves small
limits (\cite[lemma 5.1.5.5]{HT}).

\subsubsection{$\dmod$\label{sub:Defining dmod} }

We follow \cite[2.3.1]{GaRo}.  First we define the de-Rham functor
\begin{eqnarray*}
\pshv{\aff} & \xrightarrow{dR} & \pshv{\aff}\\
\mathcal{F} & \mapsto & \mathcal{F}^{dR}
\end{eqnarray*}
where the presheaf $\mathcal{F}^{dR}$ is defined by $\mathcal{F}^{dR}\left(S\right)=\mathcal{F}\left(S_{\t{red}}\right)$. 

\medskip

We define $\dmod$ on $\pshv{\aff}$ to be the composition
\[
\overbrace{\left(\pshv{\aff}\right)^{\t{op}}\xrightarrow{dR}\left(\pshv{\aff}\right)^{\t{op}}\xrightarrow{\indco}\catexl}^{\dmod:=}
\]

\begin{rem}
The functor $\pshv{\aff}\xrightarrow{\t{dR}}\pshv{\aff}$ is simply
the pullback along the composition $\aff\xrightarrow{\t{red}}\aff$,
thus is colimit preserving. It follows that $\dmod$ is limit preserving,
consequently it is equivalent to the right Kan extension 
\[
\xymatrix{\affop\ar[r]^{\t{dR}}\ar[d]^{\t j} & \pshv{\aff}^{\t{op}}\ar[r]^{\indco} & \catexl\\
\pshv{\aff}^{\t{op}}\ar@{-->}[urr]_{\dmod}
}
\]
so that for every $\mathcal{F}\in\pshv{\aff}$ we have ${\displaystyle \dmod\left(\mathcal{F}\right)\cong\lim_{S\rightarrow\mathcal{F}}\dmod\left(S\right)}$.
\end{rem}
{}
\begin{rem}
What we have defined above should rightfully be called the category
of \emph{crystalline} sheaves. It is well known that for a smooth
scheme $S$ there exist equivalences between the homotopy category
of $\dmod\left(S\right)$ (as we have defined it), and the usual derived
category of sheaves of right modules for the algebra of differential
operators on $S$ (see e.g., \cite[4.7]{GaRo}). Moreover, this equivalence
is compatible with $!$-pullback, and with Kashiwara's theorem. 
\end{rem}

\section{Appendix B}

\subsection{Limits and and colimits of adjoint diagrams}

{}~

Let 
\[
\xyR{0.5pc}\xymatrix{I\ar[r]\sp(0.4){G} & \catexl\\
i\ar@{|->}[r] & C_{i}
}
\xyR{2pc}
\]
be a small diagram. If for every morphism, $i\xrightarrow{f}j$ in
$I$, the functor $G\left(i\right)\xrightarrow{F\left(f\right)}G\left(j\right)$
admits a left adjoint%
\footnote{Not to be confused with the right adjoint it admits by virtue of being
a morphism in $\catexl$.%
}, then there exists a unique diagram (up to contractible ambiguity)
\[
\xyR{0.5pc}\xymatrix{I^{\t{op}}\ar[r]\sp(0.4){F} & \catexl\\
i\ar@{|->}[r] & C_{i}
}
\xyR{2pc}
\]
such that for every morphism, $i\xrightarrow{f}j$, the functor $C_{j}\xrightarrow{F\left(f\right)}C_{i}$
is left adjoint to $G\left(f\right)$. Let us call the pair of diagrams,
$F$ and $G$, \emph{adjoint}.

\medskip

\noindent The following lemma appears in \cite[1.3.3]{DG-DG}, where
it attributed to J. Lurie.
\begin{lem}
\label{lem:adjoint diagrams}If $F$ and $G$ are adjoint $I$-diagrams,
as above. Then \end{lem}
\begin{enumerate}
\item There exists an equivalence of stable $\infty$-categories 
\[
\underset{i\in I^{\t{op}}}{\colim}F\left(i\right)\cong\lim_{i\in I}G\left(i\right)
\]

\item Let $C\in\catexl$, and let $I^{\triangleleft}\xrightarrow{G^{\triangleleft}}\catexl$
be a co-augmentation of $C$ over of $G$, such that each $C\rightarrow C_{i}$
admits a left adjoint. Then, the natural functors 
\[
\xyR{1pc}\xymatrix{\t{colim}_{i\in I^{\t{op}}}F\left(i\right)\ar@{<->}[d]_{\cong}\ar[r]\sp(0.35){} & C\ar@{=}[d]\\
\lim_{i\in I}G\left(i\right) & C\ar[l]
}
\xyR{2pc}
\]
become adjoint, after identifying the right hand side via (1).
\item For every $j\in I$, the natural functors 
\[
\xyR{1pc}\xymatrix{C_{j}\ar[r]\sp(0.35){\rho_{j}}\ar@{=}[d] & \t{colim}_{i\in I^{\t{op}}}F\left(i\right)\ar@{<->}[d]^{\cong}\\
C_{j} & \lim_{i\in I}G\left(i\right)\ar[l]\sp(0.6){\pi_{j}}
}
\xyR{2pc}
\]
 are adjoint, after identifying the left side via 1.\end{enumerate}
\begin{proof}
{}~\end{proof}
\begin{enumerate}
\item The categories $\catexl$ and $\catexr$ both admit small limits,
and that the inclusion into $\cathat$ preserves these (\cite[5.5.3.5,5.5.3.18]{HT}
and \cite[1.1.4.4]{HA}). Consequently, since the diagram $G$ lands
in both categories (viewed as subcategories of $\catex$), we have
an equivalence 
\[
\lim\left(I\xrightarrow{G}\catexl\right)\cong\lim\left(I\xrightarrow{G}\catexr\right)
\]

There exists a duality%
\footnote{Thus, $\catexl$ and $\catexr$ admit colimits as well, but these
are not (in general) preserved by the inclusion into $\cathat$. %
} 
\[
\catexr\xrightarrow{\cong}\left(\catexl\right)^{\t{op}}
\]
which is the identity on objects, and carries each functor to its
left adjoint. It carries a limit cone $I^{\triangleleft}\xrightarrow{G^{\t{lim}}}\catexr$
for $G$, to a colimit cone $\left(I^{\t{op}}\right)^{\triangleright}\xrightarrow{F^{\t{colim}}}\catexl$
for $F$, supported on the same objects. In particular, restricting
to the cone point we get an equivalence 
\[
\lim\left(I\xrightarrow{G}\catexr\right)\cong\colim\left(I^{\t{op}}\xrightarrow{F}\catexl\right)
\]

\item In the limit and colimit diagrams above, the functors $\lim_{i\in I}G\left(i\right)\xrightarrow{\pi_{j}}C_{j}$
and $\lim_{i\in I}G\left(i\right)\xrightarrow{\pi_{j}}C_{j}$ correspond
under the duality, hence are adjoint.
\item By the same argument as in the first part of 1., the functor $C\rightarrow\lim_{i\in I}G\left(i\right)$
can be thought of as a map in $\catexr$. Thus, it admits a left adjoint,
which is its image under the duality. This dual image corresponds
to the the diagram dual to $I^{\t{\triangleleft}}\xrightarrow{G^{\triangleleft}}\catexr$,
whence the assertion follows.
\end{enumerate}

\subsection{A comparison lemma for sites}

The following lemma is an analog of the ``Comparison lemma'' \cite[Thm 2.2.3.]{Jo2},
which applies to sheaves of sets (cf. \cite[Warning 7.1.1.4]{HT}). 
\begin{lem}
\label{Comparison lemma}Let $C$ be a small category with a Grothendieck
topology, and let $C^{0}\xrightarrow[\subseteq]{j}C$ be a full subcategory.
Assume that:
\begin{enumerate}
\item $C$ admits all finite limits.
\item For any fibered product in $C$, $c_{1}\times_{c}c_{2}$, if $c_{1},c_{2}\in C^{0}$
then $c_{1}\times_{c}c_{2}\in C^{0}$ (we do not assume that $c\in C^{0}$).
\item $C^{0}$ is dense in $C$. I.e., every object in $C$ has a cover
by objects in $C^{0}$.
\end{enumerate}
Then, the restriction functor 
\[
\shv{C^{0}}{}\xleftarrow{j_{*}}\shv C{}
\]
is an equivalence, where the topology of $C^{0}$ is the pullback
of the topology of $C$. 
\end{lem}
For example, the full subcategory $\dom^{00}\subseteq\dom$ satisfies
the assumptions of this lemma.

The proof uses the slice category notation introduced in \ref{nota:slices and fibers}.
\begin{proof}
We will show that the functor given by right Kan extension 
\[
\shv{C^{0}}{}\xrightarrow{RKE_{j}}\pshv C
\]
 lands in $\shv C{}$. We will then show that the resulting adjoint
functors $\left(j_{*},RKE_{j}\right)$ 
\[
\xymatrix{\shv C{}\ar@<1ex>[r]^{j_{*}} & \shv{C^{0}}{}\ar@<1ex>[l]^{RKE_{j}}}
\]
are inverse equivalences. It is immediate that the co-unit transformation
$j_{*}\circ RKE_{j}\rightarrow1_{C^{0}}$ is an equivalence. 

We assert that the unit transformation is also an equivalence. Let
$\mathcal{F}\in\shv C{}$ , let $c\in C$, and let us prove that 
\[
\mathcal{F}\left(c\right)\rightarrow RKE_{j}\circ j_{*}\mathcal{F}\left(c\right)
\]
 is an equivalence of $\infty$-groupoids (we emphasize that $\mathcal{F}$
is assumed to be a sheaf, and not an arbitrary presheaf). It is a-priori
true that this map is an equivalence whenever $c\in C^{0}$. For general
$c\in C$, let $c^{0}\xrightarrow{f}c$ be a cover with $c^{0}\in C^{0}$,
since $C$ is assumed to admit all limits, $\mathcal{F}\left(c\right)$
may be calculated using the \v{C}ech complex associated to $f$.
By assumption (2), all the terms in this complex belong to $C^{0}$
and the assertion that the unit transformation is an equivalence follows. 

It remains to show that for every $\mathcal{F}_{0}\in\shv{C^{0}}{}$,
its right Kan extension, $\mathcal{F}:=RKE_{j}\mathcal{F}_{0}$, is
a sheaf on $C$. Let $c\in C$, and let $\mathcal{S}_{c}\subseteq C_{/c}$
be a covering sieve. We must show that 
\begin{equation}
\lim\left(\mathcal{S}_{c}^{\t{op}}\xrightarrow{\mathcal{F}}\gpd_{\infty}\right)\xleftarrow{}\mathcal{F}\left(c\right)\label{eq:sheaf condition}
\end{equation}
is an equivalence. 

The categories $C_{/c}^{0}$, and $\mathcal{S}_{c}$ are both full
subcategories of $C_{/c}$, and we denote their intersection 
\[
\mathcal{S}_{c}^{0}:=C_{/c}^{0}\cap\mathcal{S}_{c}
\]

The following triangle is a right Kan extension 
\[
\xymatrix{\left(\mathcal{S}_{c}^{0}\right)^{\t{op}}\ar[r]^{\mathcal{F}_{0}}\ar[d]_{\subseteq} & \gpd_{\infty}\\
\mathcal{S}_{c}^{\t{op}}\ar[ur]_{\mathcal{F}}
}
\]
since for every $d\rightarrow c\in\mathcal{S}_{c}$ we have that $C_{/d}^{0}\xrightarrow{\cong}\left(\mathcal{S}_{c}^{0}\right)_{/d\rightarrow c}$.
Thus it suffices to show that 
\[
\lim\left(\left(\mathcal{S}_{c}^{0}\right)^{\t{op}}\xrightarrow{\mathcal{F}_{0}}\gpd_{\infty}\right)\xleftarrow{}\lim\left(\left(C_{/c}^{0}\right)^{\t{op}}\xrightarrow{\mathcal{F}_{0}}\gpd_{\infty}\right)\cong\mathcal{F}\left(c\right)
\]
is an equivalence. In turn, the latter equivalence will follow if
we show that the following triangle is a right Kan extension
\begin{equation}
\xymatrix{\left(\mathcal{S}_{c}^{0}\right)^{\t{op}}\ar[r]^{\mathcal{F}_{0}}\ar[d]_{\subseteq} & \gpd_{\infty}\\
\left(C_{/c}^{0}\right)^{\t{op}}\ar[ur]_{\mathcal{F}_{0}}
}
\label{eq:RKE}
\end{equation}

We now use our assumptions on the relation between $C$ and $C^{0}$.
Let $c^{0}\xrightarrow{f}c$ where $c^{0}\in C^{0}$. Using hypothesis
(3), conclude that $\mathcal{S}_{c}^{0}$ generates a covering sieve,
over $c$, in $C$. It is always true that the maps 
\[
\left\{ c_{i}\times_{c}c^{0}\rightarrow c^{0}\,:\,\left(c_{i}\rightarrow c\right)\in\mathcal{S}_{c}^{0}\right\} 
\]
generate a covering sieve, over $c^{0}$, in $C$. However, according
to hypothesis (2), each of the fiber products belongs to $C^{0}$,
so that the latter maps also generate a covering sieve in $C^{0}$
(over $c_{0}$), which is simply the fibered product 

\[
\xymatrix{\left(\mathcal{S}_{c}^{0}\right)_{/\left(c^{0}\rightarrow c\right)}\ar[r]\ar[d]^{\subseteq} & \mathcal{S}_{c}^{0}\ar[d]^{\subseteq}\\
C_{/c^{0}}^{0}\ar[r]^{\circ f} & C_{/c}^{0}
}
\]
Finally, since $\mathcal{F}_{0}$ is a sheaf on $C^{0}$ we have an
equivalence 
\[
\lim\left(\left(\left(\mathcal{S}_{c}^{0}\right)_{/\left(c^{0}\rightarrow c\right)}\right)^{\t{op}}\xrightarrow{\mathcal{F}_{0}}\gpd_{\infty}\right)\xleftarrow{\cong}\mathcal{F}_{0}\left(c^{0}\right)
\]
implying that \ref{eq:RKE} is a right Kan extension. Tracing back,
we conclude that $RKE_{j}\mathcal{F}_{0}\in\pshv C$ is a sheaf.
\end{proof}

\subsection{The Ran space }
\begin{prop}
\label{lem:ran fibd in sets}The functor of points $\aff^{\t{op}}\xrightarrow{\ran}\gpd_{\infty}$
takes values in sets. Namely, for every $S\in\aff$. 
\[
\ran\left(S\right)=\left\{ F\subseteq\t{Hom}\left(S,X\right)\,:\, F\,\mbox{ finite,\,\ non-empty}\right\} 
\]
\end{prop}
\begin{proof}
Consider the augmented $\finsurop$ diagram 
\[
\xymatrix{\left(\finsur\cup\left\{ \emptyset\right\} \right)^{\t{op}}\ar[r] & \set\subseteq\gpdinf}
\]
given by 
\[
\xymatrix{\t{Hom}\left(S,X\right)\ar[r]\ar[drr] & \left(\t{Hom}\left(S,X\right)\right)^{2}\ar@<1ex>[r]\ar@<-1ex>[r]\ar@(ul,ur)^{S_{2}}\ar[drr]-<1.3cm,-0.2cm> & \left(\t{Hom}\left(S,X\right)\right)^{3}\ar@<2ex>[r]\ar@<-2ex>[r]\ar[r]\ar@(ul,ur)^{S_{3}}\ar[dr]-<0.8cm,-0.2cm> & \cdots\ar@<2ex>[d]^{\cdots}\\
 &  &  & \save+<0cm,-0.2cm>*{\left\{ F\subseteq\t{Hom}\left(S,X\right):F\,\mbox{ finite}\right\} }\restore
}
\]
(the circular arrows represent the action of the respective symmetric
groups on $n$ elements, $S_{n}$). By definition, $\ran\left(S\right)$
is the colimit in $\gpd_{\infty}$%
\footnote{We emphasize that we want to show this diagram is a homotopy colimit.
That this is a colimit diagram in sets is obvious.%
} of the top row, so we must show that diagram is a colimit diagram.

It suffices to prove that for every $F\in\left\{ F\subseteq\t{Hom}\left(S,X\right)\,:\, F\,\mbox{ finite}\right\} $,
the following homotopy fiber is contractible 
\[
\xymatrix{\t{pt}\times_{\left\{ F\subseteq\t{Hom}\left(S,X\right)\,:\, F\,\mbox{ finite}\right\} }\ran\left(S\right)\ar[r]\ar[d] & \ran\left(S\right)\ar[d]\\
\t{pt}\ar[r]_{F} & \left\{ F\subseteq\t{Hom}\left(S,X\right)\,:\, F\,\mbox{ finite}\right\} 
}
\]
 Since colimits in $\gpdinf$ are universal, this fiber is the colimit
of the $\finsur^{\t{op}}$ diagram in $\gpdinf$ 
\[
\xymatrix{\t{Surj}\left(\left\{ 1\right\} ,F\right)\ar[r] & \t{Surj}\left(\left\{ 2\right\} ,F\right)\ar@<1ex>[r]\ar@<-1ex>[r]\ar@(ul,ur)^{S_{2}} & \t{Surj}\left(\left\{ 3\right\} ,F\right)\ar@<2ex>[r]\ar@<-2ex>[r]\ar[r]\ar@(ul,ur)^{S_{3}} & \cdots}
\]
where $\left\{ n\right\} $ denotes a finite set with $n$ elements
and $\t{Surj}\left(\left\{ n\right\} ,F\right)$ is the set of surjections
$\left\{ n\right\} \twoheadrightarrow F$. We prove that this colimit
is contractible. Applying the Grothendieck un-straightening construction,
we get the Cartesian fibration 
\[
\xymatrix{\left(\finsur^{\t{op}}\right)_{/F}\ar[d]\\
\finsurop
}
\]
(where $F$ is now considered as an abstract finite set). The homotopy
type we are after is the the weak homotopy type of the total space,
$\left(\finsur^{\t{op}}\right)_{/F}$, which is evidently contractible
since it has a terminal element. \end{proof}

\bibliographystyle{amsalpha}
\bibliography{Thesis}

\def\cprime{$'$}
\providecommand{\bysame}{\leavevmode\hbox to3em{\hrulefill}\thinspace}
\providecommand{\MR}{\relax\ifhmode\unskip\space\fi MR }
\providecommand{\MRhref}[2]{%
  \href{http://www.ams.org/mathscinet-getitem?mr=#1}{#2}
}
\providecommand{\href}[2]{#2}
\begin{thebibliography}{BFGM02}

\bibitem[BD04]{CA}
Alexander Beilinson and Vladimir Drinfeld, \emph{Chiral algebras}, American
  Mathematical Society Colloquium Publications, vol.~51, American Mathematical
  Society, Providence, RI, 2004. \MR{2058353 (2005d:17007)}

\bibitem[BFGM02]{BFGM}
A.~Braverman, M.~Finkelberg, D.~Gaitsgory, and I.~Mirkovi{\'c},
  \emph{Intersection cohomology of {D}rinfeld's compactifications}, Selecta
  Math. (N.S.) \textbf{8} (2002), no.~3, 381--418. \MR{1931170 (2003h:14060)}

\bibitem[BG02]{BG}
A.~Braverman and D.~Gaitsgory, \emph{Geometric {E}isenstein series}, Invent.
  Math. \textbf{150} (2002), no.~2, 287--384. \MR{1933587 (2003k:11109)}

\bibitem[DS95]{DS}
V.~G. Drinfel{\cprime}d and Carlos Simpson, \emph{{$B$}-structures on
  {$G$}-bundles and local triviality}, Math. Res. Lett. \textbf{2} (1995),
  no.~6, 823--829. \MR{1362973 (96k:14013)}

\bibitem[FM99]{FM}
Michael Finkelberg and Ivan Mirkovi{\'c}, \emph{Semi-infinite flags. {I}.
  {C}ase of global curve {$\bold P^1$}}, Differential topology,
  infinite-dimensional {L}ie algebras, and applications, Amer. Math. Soc.
  Transl. Ser. 2, vol. 194, Amer. Math. Soc., Providence, RI, 1999,
  pp.~81--112. \MR{1729360 (2001j:14029)}

\bibitem[Gai10a]{DG-Ran}
Dennis Gaitsgory, \emph{Categories over the {R}an space},
  \url{http://math.harvard.edu/~gaitsgde/GL/Ran.pdf}, 2010, unpublished.

\bibitem[Gai10b]{DG-Whit}
\bysame, \emph{The extended {W}hittaker category},
  \url{http://math.harvard.edu/~gaitsgde/GL/extWhit.pdf}, 2010, unpublished.

\bibitem[Gai10c]{DG-Indcoh}
\bysame, \emph{Ind-coherent sheaves},
  \url{http://math.harvard.edu/~gaitsgde/GL/IndCohtext.pdf}, 2010, preprint.

\bibitem[Gai11a]{DG-Cont}
\bysame, \emph{Contractibility of the space of rational maps},
  \url{http://math.harvard.edu/~gaitsgde/GL/Contractibility.pdf}, 2011,
  preprint.

\bibitem[Gai11b]{DG-DG}
\bysame, \emph{Notes on geometric langlands: Generalities on dg categories},
  \url{http://math.harvard.edu/~gaitsgde/GL/textDG.pdf}, 2011, unpublished.

\bibitem[GR11]{GaRo}
Dennis Gaitsgory and Nick Rozenblyum, \emph{Crystals and d-modules},
  \url{http://math.harvard.edu/~gaitsgde/GL/IndCohtext.pdf}, 2011, preprint.

\bibitem[Har77]{Hart}
Robin Hartshorne, \emph{Algebraic geometry}, Springer-Verlag, New York, 1977,
  Graduate Texts in Mathematics, No. 52. \MR{0463157 (57 \#3116)}

\bibitem[Joh02]{Jo2}
Peter~T. Johnstone, \emph{Sketches of an elephant: a topos theory compendium.
  {V}ol. 2}, Oxford Logic Guides, vol.~44, The Clarendon Press Oxford
  University Press, Oxford, 2002. \MR{2063092 (2005g:18007)}

\bibitem[Lur09]{HT}
Jacob Lurie, \emph{Higher topos theory}, Annals of Mathematics Studies, vol.
  170, Princeton University Press, Princeton, NJ, 2009. \MR{2522659
  (2010j:18001)}

\bibitem[Lur11a]{HA}
\bysame, \emph{Higher algebra},
  \url{http://math.harvard.edu/~lurie/papers/higheralgebra.pdf}, 2011,
  preprint.

\bibitem[Lur11b]{V}
\bysame, \emph{Structured {S}paces},
  \url{http://math.harvard.edu/~lurie/papers/DAG-V.pdf}, 2011, preprint.

\end{thebibliography}

\end{document}